\documentclass[11pt]{amsart}
\usepackage{amsmath,amsthm,amssymb}
\usepackage{mathtools} 
\usepackage{dsfont}   

\usepackage{cite}

\usepackage{graphicx}
\usepackage{epstopdf} 
\usepackage{caption}
\usepackage{subfig}
\usepackage{tikz}

\usepackage{algorithm}
\usepackage{algpseudocode}

\usepackage{booktabs}
\usepackage{multirow}

\usepackage{xcolor}
\usepackage{tcolorbox}
\tcbuselibrary{theorems}
\newtcbtheorem[number within=section]{permstrategy}{Permutation Strategy}{%
    colback=white,
    colframe=black,
    colbacktitle=white,
    coltitle=black,
    fonttitle=\bfseries\large,
    boxrule=0.5pt,
    titlerule=0.25pt,
    top=1pt,
    bottom=1pt,
    left=6pt,
    right=6pt
}{strat}

\usepackage[colorlinks, citecolor=red]{hyperref}

\newtheorem{definition}{Definition}[section]

\newtheorem{prop}[definition]{Proposition}
\newtheorem{theorem}[definition]{Theorem}
\newtheorem{lemma}[definition]{Lemma}

\newtheorem{remark}[definition]{Remark}

\date{}

\begin{document}
\baselineskip 18pt
\title[Enhanced RDR method]{Enhanced randomized Douglas-Rachford method for linear systems: Improved probabilities and adaptive momentum }

\author{Liqi Guo}
\address{School of Mathematical Sciences, Beihang University,  Beijing, 100191, China.}
\email{glq2023@buaa.edu.cn}

\author{Ruike Xiang}
\address{School of Mathematical Sciences, Beihang University,  Beijing, 100191, China.}
\email{xiangrk@buaa.edu.cn}

\author{Deren Han}
\address{LMIB of the Ministry of Education, School of Mathematical Sciences, Beihang University, Beijing, 100191, China. }
\email{handr@buaa.edu.cn}

\author{Jiaxin Xie}
\address{LMIB of the Ministry of Education, School of Mathematical Sciences, Beihang University, Beijing, 100191, China. }
\email{xiejx@buaa.edu.cn}

\begin{abstract}
Randomized iterative methods have gained recent interest in machine learning and signal processing for solving large-scale linear systems. One such example is the randomized Douglas–Rachford (RDR) method, which updates the iterate by reflecting it through two randomly selected hyperplanes and taking a convex combination with the current point. In this work, we enhance RDR by introducing improved sampling strategies and an adaptive heavy-ball momentum scheme. Specifically, we incorporate without-replacement sampling and volume sampling into RDR and establish linear convergence in expectation for both schemes. For without-replacement sampling, we further establish a larger one-step descent than that of conventional RDR. For volume sampling, the convergence rate explicitly incorporates geometric information, and a tighter comparison with RDR is obtained under additional structural assumptions. Furthermore, we develop an adaptive momentum mechanism that determines the relaxation and momentum parameters directly from the iteration information without manual tuning. We establish linear convergence in expectation for the resulting method. Numerical experiments demonstrate the effectiveness of the proposed methods across a range of problem settings.
\end{abstract}

\maketitle

\let\thefootnote\relax

\let\thefootnote\relax

\section{Introduction}
Solving linear systems is a fundamental problem in numerous fields, including statistics \cite{hastie2009elements}, scientific computing \cite{golub2013matrix}, and machine learning \cite{boyd2018introduction,bishop2006pattern}. With the advent of the big data era, there is an increasing demand for efficient algorithms capable of handling linear systems of unprecedented scale. Recent research has focused on developing randomized iterative methods \cite{needell2014stochasticMP,necoara2019faster,strohmer2009randomized,han2024randomized,zeng2024adaptive,Gow15,xie2024randomized}, which offer advantages such as low per-iteration cost, optimal complexity, ease of implementation, and strong empirical performance on large-scale problems.

Among these methods, the randomized Kaczmarz (RK) algorithm is one of the most widely studied. The original Kaczmarz method \cite{karczmarz1937angenaherte} cyclically projects the current iterate onto the solution space of individual linear constraints and converges to a solution of consistent systems. Empirical studies have shown that using the rows of the coefficient matrix  in a random order, rather than a fixed deterministic one, can significantly accelerate convergence \cite{herman1993algebraic,natterer2001mathematics,feichtinger1992new}. Strohmer and Vershynin \cite{strohmer2009randomized} formalized this idea by introducing a randomized variant in which rows are sampled with probabilities proportional to their squared norms, and proved that the RK method converges linearly in expectation for consistent systems.
Due to their simplicity, efficiency, and theoretical guarantees, Kaczmarz-type methods have found widespread applications in areas such as phase retrieval \cite{tan2019phase,huang2022linear}, compressed sensing \cite{schopfer2019linear,yuan2022adaptively,jeong2023linear}, ridge regression \cite{hefny2017rows,gazagnadou2022ridgesketch},  tensor recovery \cite{chen2021regularized,ma2022randomized}, adversarial optimization \cite{huang2024randomized}, and absolute value equations \cite{xie2024randomized}.

A recent addition to this family of randomized methods is the randomized Douglas-Rachford (RDR) algorithm \cite{han2024randomized}, which introduces randomization techniques into the classical Douglas-Rachford (DR) splitting method \cite{lindstrom2021survey,Ara14,Li16,Eck92} for solving feasibility problems arising from linear systems. Indeed, the motivation for studying RDR stems from the operator-splitting structure inherited from the classical DR method. In particular, the DR framework naturally accommodates a two-operator splitting structure, allowing the corresponding resolvents or projections to be handled separately. This feature makes it potentially applicable beyond linear systems, including feasibility problems, nonsmooth composite optimization, and constrained optimization. In contrast, RK is specifically tailored to projections onto individual linear equations and does not directly possess such an operator-splitting structure. Therefore, the study of RDR for linear systems provides a natural first step toward understanding randomized DR schemes and their potential extensions to more general operator-splitting settings. By introducing stochasticity into the iterative process, RDR aims to combine the robustness of operator splitting schemes with the efficiency of randomized updates. Preliminary studies demonstrate that RDR achieves promising convergence behavior and strong empirical performance. Notably, it converges linearly in expectation, with a convergence rate that is independent of the dimension of the coefficient matrix.

An evident limitation of the RDR method is its probability criterion for selecting the hyperplane for reflecting the current iterate.
The current approach employs a with-replacement sampling strategy \cite{han2024randomized}, which simplifies convergence analysis but may repeatedly select certain rows, particularly when the row norms of the coefficient matrix vary significantly. This redundancy can reduce the effectiveness of reflective projections and slow convergence. In addition, current methods do not adequately exploit the angular relationships between hyperplanes when selecting the projection direction. 
Moreover, Polyak's heavy-ball momentum \cite{polyak1964some} has also been incorporated into the RDR framework \cite{han2024randomized} to accelerate convergence. While this momentum-based approach can enhance empirical performance, it introduces additional hyperparameters, namely the step size and momentum coefficient, which require careful tuning in practice.   This tuning process is often problem-dependent and may compromise the robustness and ease of deployment of the method.

In this paper, we enhance the RDR framework by proposing the  improved RDR (IRDR)
method, which improves the probabilistic selection of projection hyperplanes. Depending on the structural properties of the coefficient matrix, { IRDR} flexibly employs either without-replacement sampling or volume sampling \cite{rodomanov2020randomized,macchi1975coincidence,xiang2025randomized} as the selection criterion. The former mitigates redundancy by avoiding repeated selection of reflection hyperplanes, while the latter further incorporates geometric information, specifically the angular relationships between hyperplanes, thereby improving the quality of reflections. Notably, when the row norms of the matrix differ significantly, { IRDR} can substantially enhance the efficiency of the iterative process. 
Beyond their potential numerical benefits, these sampling strategies also introduce an additional theoretical challenge due to the dependence between the selected hyperplanes. The original RDR analysis relies on their independence, which allows the expectation of the composition of reflection operators to be factorized. Under without-replacement sampling and volume sampling, however, the selected indices are dependent, and such a factorization is no longer directly applicable. Establishing convergence under these dependent sampling mechanisms therefore constitutes an additional theoretical contribution of this work.

Furthermore, inspired by recent work \cite{zeng2024adaptive,sun2025connecting}, we further propose the adaptive momentum  RDR (AmRDR)
method, which incorporates a dynamic variant of heavy-ball momentum. In this scheme, both the step size and momentum parameter are adaptively adjusted at each iteration based on the previous iterates. Unlike the fixed-parameter mRDR method, whose performance can depend strongly on problem-specific choices of the relaxation and momentum parameters, AmRDR determines these parameters directly from the iterates, thereby reducing the need for problem-dependent parameter tuning. This feature is particularly advantageous when repeated parameter tuning is computationally expensive or impractical.
Our theoretical analysis confirms that AmRDR converges linearly in expectation with an improved upper bound. Numerical experiments demonstrate that incorporating 
improved probabilistic criteria and adaptive momentum techniques into the RDR method can significantly enhance its accuracy and efficiency.

The remainder of the paper is organized as follows.
After introducing some preliminaries in Section 2,  we present and analyze our {IRDR} method in Section 3. 
In Section 4, we propose {AmRDR} and show its linear convergence rate. 	In Section 5, we demonstrate numerical experiments that illustrate the effectiveness of algorithm design. Finally, we conclude the paper in Section 6. Proofs of all main results are provided in the appendix.

\section{Preliminaries}
	
	\subsection{Notations}
	Throughout the paper, for any random variables $\xi$ and $\zeta$, we use $\mathbb{E}[\xi]$ and $\mathbb{E}[\xi\lvert \zeta]$ to denote the expectation of $\xi$ and the conditional expectation of $\xi$ given $\zeta$. For any integer $s$ satisfying \(1 \leq s \leq m\), we use \(\binom{[m]}{s}\) to denote the collection of all \(s\)-element subsets of \([m] := \{1, \ldots, m\}\), and use $[m]^s:=[m]\times\cdots\times[m]$ to denote the set of all ordered $s$-tuples $(i_1,\ldots,i_s)$ with $i_\ell\in[m]$ for $\ell=1,\ldots,s$.
	For a vector $x\in\mathbb{R}^n$, we use $x_i$, $x^\top$, and $\|x\|_2$ to denote the $i$-th element, the transpose, and the Euclidean norm of $x$, respectively. We use $\text{diag}(x)$ to denote the diagonal matrix whose entries on the diagonal are the components of $x$.

	For any matrix $A\in\mathbb{R}^{m\times n}$, we use $a^\top_{i}$, $A^\top$, $A^{\dagger}$, $ \|A\|_2 $, $ \|A\|_F $, $\operatorname{rank}(A)$, and $\mbox{Range}(A)$ to denote the $i$-th row, the transpose, the Moore-Penrose pseudoinverse, the spectral norm, the Frobenius norm, the rank, and the range space of $A$, respectively. For $\mathcal S=(i_1,\ldots,i_s)\in[m]^s$, we use $A_{\mathcal S}$ to denote the matrix whose $\ell$-th row is $a_{i_\ell}^{\top}$.
	The nonzero singular values of a matrix $A$ are $\sigma_1(A)\geq\sigma_2(A)\geq\cdots\geq\sigma_{r}(A):=\sigma_{\min}(A)>0$, where $r$ is the rank of $A$
	and $\sigma_{\min}(A)$ denotes the smallest nonzero singular values of $A$. 
	A symmetric matrix $A\in\mathbb{R}^{n\times n}$ is said to be positive definite if $x^\top A x> 0$ holds for every nonzero vector
	$x\in\mathbb{R}^n$.

    \begin{definition}[Volume sampling] 
        \label{vs}
        Let \( A \) be an \( m \times n \) matrix and suppose the integer \( s \) satisfies \( 1 \leq s \leq \operatorname{rank}(A) \). Consider a random variable \( \mathcal{S}_0=(i_1,\ldots,i_s) \) that takes values in \( [m]^s \). We define \( \mathcal{S}_0 \) as being generated according to \( s \)-element volume sampling with respect to \( AA^\top \), denoted by \( \mathcal{S}_0 \sim \operatorname{Vol}_s(AA^\top) \), if for all \( \mathcal{S} \in [m]^s \), the probability that \( \mathcal{S}_0 \) equals \( \mathcal{S} \) is given by
        \[
        \mathbb{P}(\mathcal{S}_0 = \mathcal{S}) = \frac{\operatorname{det}(A_{\mathcal{S}} A_{\mathcal{S}}^\top)}{\sum_{\mathcal{J} \in [m]^s} \operatorname{det}(A_{\mathcal{J}} A_{\mathcal{J}}^\top)}.
        \]
    \end{definition}
    
    Since $A_{\mathcal S}A_{\mathcal S}^{\top}$ is positive semidefinite, we have
    $
    \det(A_{\mathcal S}A_{\mathcal S}^{\top})\geq 0.
    $
    Moreover,
    $
    \det(A_{\mathcal S}A_{\mathcal S}^{\top})=0
    $
    whenever the rows of $A_{\mathcal S}$ are linearly dependent. Hence, any ordered tuple $\mathcal S=(i_1,\ldots,i_s)$ for which $A_{\mathcal S}$ is rank deficient has zero sampling probability. Consequently, any tuple sampled   satisfies
    \[
    \operatorname{rank}(A_{\mathcal S})=s.
    \]
    In particular, its indices are pairwise distinct, i.e.,
    $
    i_j\neq i_k$   for all $ j\neq k,
   $
    and none of the selected rows is zero, i.e.,
    $
    a_{i_j}\neq 0,j=1,\ldots,s.
    $
    Finally, since $s\leq\operatorname{rank}(A)$, there exist $s$ linearly independent rows of $A$. Therefore, the normalizing denominator in the definition is strictly positive.

\subsection{The randomized Douglas-Rachford method}
	
Consider the consistent linear system
\begin{equation}
	\label{linear system}
	Ax = b, \ \ A\in\mathbb{R}^{m\times n}, \ b\in\mathbb{R}^m,
\end{equation}
where each row defines a hyperplane 
\(\mathcal{H}_i := \{ x \in \mathbb{R}^n \mid  \langle a_i, x \rangle = b_i \}.\)
We use \( \mathcal{P}_{\mathcal{H}_i} \) and \( \mathcal{R}_{\mathcal{H}_i} \) to denote the orthogonal projection and reflection operators with respect to \( \mathcal{H}_i \), respectively. Specifically, for any \( x \in \mathbb{R}^n \), these operators are given by
\[
\mathcal{P}_{\mathcal{H}_i}(x) = x - \frac{\langle a_i, x \rangle - b_i}{\|a_i\|_2^2} a_i \ \ \text{and} \ \
\mathcal{R}_{\mathcal{H}_i}(x) = x - 2\frac{\langle a_i, x \rangle - b_i}{\|a_i\|_2^2} a_i.
\]
Starting from an initial point \( x^0 \in \mathbb{R}^n \), the randomized Kaczmarz (RK) method \cite{strohmer2009randomized} updates the iterates via
\[
x^{k+1} = \mathcal{P}_{\mathcal{H}_{i_k}}(x^k) = x^k - \frac{\langle a_{i_k}, x^k \rangle - b_{i_k}}{\|a_{i_k}\|_2^2} a_{i_k},
\]
where \( \mathbb{P}(i_k = i) = \frac{\|a_i\|_2^2}{\|A\|_F^2} \). This procedure is illustrated in Figure~\ref{figueRK}.
In contrast, the randomized Douglas–Rachford (RDR) method \cite{han2024randomized} generates iterates via
\[
x^{k+1} = \frac{1}{2} \left(I + \mathcal{R}_{\mathcal{H}_{i_{k_2}}} \mathcal{R}_{\mathcal{H}_{i_{k_1}}} \right)(x^k),
\]
where the indices \( i_{k_1} \) and \( i_{k_2} \) are independently sampled from the same distribution, i.e.,
$
\mathbb{P}(i_{k_1} = i) = \mathbb{P}(i_{k_2} = i) = \frac{\|a_i\|_2^2}{\|A\|_F^2}.
$
The Douglas–Rachford (DR) method is illustrated in Figure~\ref{figueRDR}.

				\begin{figure}[hptb]
		\centering
        \subfloat[Kaczmarz
			\label{figueRK}]{
    \begin{minipage}{0.5\textwidth}
			\centering
		\begin{tikzpicture}
	\draw (-2,0) -- (2,0);
	\draw (-1.5,2) -- (1.2,-1.6);
	\draw (-1.2,-1.6) -- (1.5,2);
	
	\filldraw (0,0) circle [radius=1.5pt]
	(1.2,1.6) circle [radius=1.5pt]
	(1.2,0) circle [radius=1.5pt]
	(0.432,-0.576) circle [radius=1.5pt];
	
	\draw (-0.5,0.25) node {$x^*$}
	(1.55,1.6) node {$x^k$}
	(1.75,0.25) node {$x^{k+1}$}
	(1.1,-0.6) node {$x^{k+2}$};
	\draw (-1.5,0.25) node {$\mathcal{H}_2$}
	(-1.15,-1.05) node {$\mathcal{H}_1$}
	(-1.25,1.2) node {$\mathcal{H}_3$};
	
	\draw [-stealth] (1.2,1.6) -- (1.2,0.1);
	\draw [-stealth] (1.2,0) -- (0.512,-0.516);
\end{tikzpicture}

    \end{minipage}\hfill}
		\subfloat[Douglas-Rachford
			\label{figueRDR}]{
        \begin{minipage}{0.5\textwidth}
			\centering
			\begin{tikzpicture}
			\draw (-2,0) -- (2,0);
			\draw (-1.5,2) -- (1.2,-1.6);
			\draw (-1.2,-1.6) -- (1.5,2);
			
			\filldraw (0,0) circle [radius=1.5pt]
			(0.5,2) circle [radius=1.5pt]
			(1.78,1.025) circle [radius=1.5pt]
			(1.78,-1.025) circle [radius=1.5pt]
			(1.14,0.4875) circle [radius=1.5pt];
			
			\draw (-0.5,0.25) node {$x^*$}
			(0.9,2.1) node {$x^k$}
			(3.05,1.125) node {$y^k=\mathcal{R}_{\mathcal{H}_1}(x^k)$}
			(3.05,-1.025) node {$z^k=\mathcal{R}_{\mathcal{H}_2}(y^k)$}
			(0.75,0.3) node {$x^{k+1}$};
			\draw (-1.5,0.25) node {$\mathcal{H}_2$}
			(-1.15,-1.05) node {$\mathcal{H}_1$}
			(-1.25,1.2) node {$\mathcal{H}_3$};
			
			\draw [-stealth] (0.5,2) -- (1.7,1.085);
			\draw [-stealth] (1.78,1.025) -- (1.78,-0.925);
			\draw [dash pattern=on 3pt off 2pt] (0.5,2) -- (1.78,-1.025);
		\end{tikzpicture}

		\end{minipage}}
        
		\caption{Geometric interpretations of the Kaczmarz method and the Douglas-Rachford method.}
		\label{fig:rk-rdr}
	\end{figure}
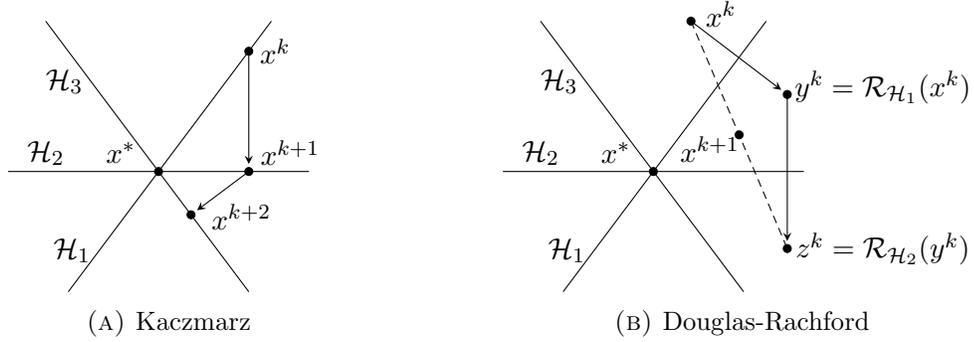

While effective, this independent sampling strategy may frequently select the same hyperplane, particularly when the row norms of \( A \) vary significantly. Moreover, the method does not exploit the geometric relationships between hyperplanes, which could otherwise inform more effective projection directions. In this paper, we aim to develop improved sampling strategies that reduce redundancy and incorporate geometric information to accelerate convergence.

Finally, we note that the RDR method was shown in \cite[Theorem 3.1]{han2024randomized} to converge linearly in expectation:
\begin{equation}\label{xie-RDR-rate}
	\mathbb{E}\left[\|x^{k} - x_*^0\|_2^2 \right] \leq \left( \frac{1}{2} + \frac{1}{2} \left(1 - 2\frac{\sigma_{\min}^2(A)}{\|A\|_F^2} \right)^2 \right)^k \|x^0 - x_*^0\|_2^2,
\end{equation}
where \( x_*^0 := A^\dagger b + (I - A^\dagger A)x^0 \) denotes the projection of \( x^0 \) onto the solution set \( \{x \in \mathbb{R}^n \mid Ax = b\} \).   Throughout this paper, we assume that \( \operatorname{rank}(A) \geq 2 \). This assumption is necessary because, as shown in Figure~\ref{figueRDR}, when \( \operatorname{rank}(A) = 1 \), the DR iteration reduces to repeated reflections across the same hyperplane, which leads to a stationary sequence \( x^0 = x^1 = \cdots = x^k \). In such cases, the method fails to make progress. In fact, when \( \operatorname{rank}(A) = 1 \), the solution can be obtained in a single projection step.
	
	\subsection{The heavy-ball momentum method}
	
Consider the unconstrained minimization problem
\[
\min_{x \in \mathbb{R}^n} f(x),
\]
where \( f \) is a differentiable convex function.  
Gradient descent (GD) is a fundamental method for solving this problem, with the iterative update
\[
x^{k+1} = x^k - \alpha \nabla f(x^k),
\]
where \( \alpha > 0 \) is the stepsize.
To accelerate convergence, Polyak \cite{polyak1964some} introduced a momentum term \( \beta(x^k - x^{k-1}) \), leading to the heavy-ball method:
\[
x^{k+1} = x^k - \alpha \nabla f(x^k) + \beta(x^k - x^{k-1}).
\]
Motivated by this idea, momentum techniques have been successfully incorporated into the RDR framework \cite{han2024randomized}. Although the momentum variant of RDR can achieve linear convergence with suitable choices of the stepsize \( \alpha \) and momentum parameter \( \beta \) \cite[Theorem 4.1]{han2024randomized}, these parameters are often problem-specific and require careful tuning,  limiting the robustness and ease of use of the method. In this paper,  we develop an adaptive momentum RDR variant inspired by recent work \cite{zeng2024adaptive,sun2025connecting}, where both \(\alpha\) and \(\beta\) are dynamically adjusted at each iteration based on historical information from previous iterates.

\section{RDR with improved sampling strategies}
\label{section3}

In this section, we present two improved sampling strategies that aim to enhance the efficiency and convergence behavior of the RDR method. The first is without-replacement sampling, which avoids selecting the same index more than once within a single iteration. The second is volume sampling, which leverages geometric information to promote diversity in the selected subsets.

\textbf{Without-replacement sampling.} At the $k$-th iteration, the indices $i_{k_1}$ and $i_{k_2}$ are selected according to the following scheme:
\[
\boxed{
	\text{\bf Strategy I:} \quad
	\mathbb{P}(i_{k_1} = i) = \frac{\Vert a_i \Vert_2^2}{\Vert A \Vert_F^2}, \quad
	\mathbb{P}(i_{k_2} = j \ | \ i_{k_1} = i) = \frac{\Vert a_j \Vert_2^2}{\Vert A \Vert_F^2 - \Vert a_{i} \Vert_2^2} (j\neq i).
}
\]
This strategy effectively avoids repeatedly selecting the same reflection hyperplane, particularly when there is a significant disparity in the row norms of matrix~$A$, thereby improving the algorithm's overall efficiency.

\textbf{Volume sampling.} In this approach, we select an ordered pair of indices $\mathcal{S}_k = ( i_{k_1}, i_{k_2} )$ according to the volume sampling distribution:
\[
\boxed{
	\text{\bf Strategy II:} \quad \mathcal{S}_k \sim \operatorname{Vol}_2(AA^\top).
}
\]
Volume sampling favors row selections that span a larger volume, thus encouraging the selection of geometrically diverse directions, which can lead to faster convergence in practice.

 The canonical DR algorithm has inspired numerous modifications and relaxations in the literature. One particularly important and widely used variant is the generalized Douglas-Rachford method, introduced by Eckstein and Bertsekas~\cite{Eck92}, which incorporates a relaxation parameter $\alpha \in (0,1)$:
    \[
    x^{k+1} = \left((1 - \alpha) I + \alpha \mathcal{R}_{\mathcal{H}_{i_{k_2}}} \mathcal{R}_{\mathcal{H}_{i_{k_1}}} \right)(x^k).
    \]
    This formulation reduces to the standard DR method when $\alpha = \frac{1}{2}$. The introduction of the relaxation parameter $\alpha$ is significant, as it can lead to practical acceleration of convergence, especially when $\alpha > \frac{1}{2}$~\cite{strohmer2009randomized,han2024randomized}. In this paper, we also consider the use of relaxation, and we present our improved RDR method in Algorithm~\ref{alg:IRDR}.

	\begin{algorithm}[htpb]
		\caption{{ Improved randomized Douglas-Rachford (IRDR) method}}
		\label{alg:IRDR}
		\begin{algorithmic}
			\Require 
			$A \in \mathbb{R}^{m\times n}, b \in \mathbb{R}^{m}, k=0, $
			 relaxation parameter $\alpha \in (0,1)$, and initial vector $x^0 \in \mathbb{R}^{n}$.
             \Ensure  
			The approximate solution $x^k$.
		
			\item[1:] Select $(i_{k_1},i_{k_2})$ according to Strategy I (sampling without replacement) or Strategy II (volume sampling). 
			\item[2:] Compute 
			$$
			\left\{
			\begin{array}{l}
				y^k =\mathcal{R}_{\mathcal{H}_{i_{k_1}} }(x^k)  
					=x^k - 2\frac{\langle a_{i_{k_1}}, x^k \rangle - b_{i_{k_1}}}{\Vert a_{i_{k_1}} \Vert_2^2} a_{i_{k_1}},\\
				 z^k =\mathcal{R}_{\mathcal{H}_{i_{k_2}} }(y^k)  = y^k - 2\frac{\langle a_{i_{k_2}}, y^k \rangle - b_{i_{k_2}}}{\Vert a_{i_{k_2}} \Vert_2^2} a_{i_{k_2}}.
			\end{array}\right.
			$$
			\item[3:] Update $$ 
			x^{k+1} =(1-\alpha)x^k+\alpha z^k.
			$$
			
			\item[4:] If the stopping rule is satisfied, stop, and go to output. Otherwise, set $k = k + 1$, and return to step 1.
		\end{algorithmic}
	\end{algorithm}
	
\subsection{Convergence Analysis}

To establish the convergence of Algorithm~\ref{alg:IRDR}, we begin by introducing some notations.   Define the constant
\(
\Delta \coloneqq \sum_{j=1}^m \frac{\|a_j\|_2^2}{\|A\|_F^2 - \|a_j\|_2^2},
\)
and consider the matrix \( M \in \mathbb{R}^{m \times m} \). To make its dependence on the row norms and pairwise inner products explicit, we retain the following displayed form:
\begin{equation}\label{definition-M}
M := (\Delta + 1)I - 
\begin{bmatrix}
	\frac{\|a_1\|_2^2}{\|A\|_F^2 - \|a_1\|_2^2} & \frac{2\langle a_1, a_2 \rangle}{\|A\|_F^2 - \|a_2\|_2^2} & \cdots & \frac{2\langle a_1, a_m \rangle}{\|A\|_F^2 - \|a_m\|_2^2} \\
	\frac{2\langle a_1, a_2 \rangle}{\|A\|_F^2 - \|a_1\|_2^2} & \frac{\|a_2\|_2^2}{\|A\|_F^2 - \|a_2\|_2^2} & \cdots & \frac{2\langle a_2, a_m \rangle}{\|A\|_F^2 - \|a_m\|_2^2} \\
	\vdots & \vdots & \ddots & \vdots \\
	\frac{2\langle a_1, a_m \rangle}{\|A\|_F^2 - \|a_1\|_2^2} & \frac{2\langle a_2, a_m \rangle}{\|A\|_F^2 - \|a_2\|_2^2} & \cdots & \frac{\|a_m\|_2^2}{\|A\|_F^2 - \|a_m\|_2^2}
\end{bmatrix}.
\end{equation}
Equivalently, the entries of $M$ are
\[
M_{ij}=
\begin{cases}
	\displaystyle
	\Delta+1-\frac{\|a_i\|_2^2}{\|A\|_F^2-\|a_i\|_2^2},
	& i=j,\\[8pt]
	\displaystyle
	-\frac{2\langle a_i,a_j\rangle}
	{\|A\|_F^2-\|a_j\|_2^2},
	& i\neq j.
\end{cases}
\] 
	Any square matrix admits a unique decomposition into its symmetric part and skew-symmetric part. Accordingly, we define the symmetric part $ S_M \in \mathbb{R}^{m \times m}$
	of $M$ as
	\begin{equation}\label{sym-skew-decomp}
		S_M := \frac{M + M^\top}{2}.
	\end{equation}
For any \( i, j \in [m] \), we define
$
g_{i,j} := 1 - \frac{\langle a_i, a_j \rangle^2}{\|a_i\|_2^2 \|a_j\|_2^2}.
$
	Define $N\in\mathbb R^{m\times m}$ componentwise by
	\[
	N_{ij}=
	\begin{cases}
		\displaystyle
		\sum_{\ell=1}^{m}g_{i,\ell}\|a_\ell\|_2^2,
		& i=j,\\[8pt]
		\displaystyle
		-g_{i,j}\langle a_i,a_j\rangle,
		& i\neq j.
	\end{cases}
	\]
	Its matrix form is 
\begin{equation}\label{definition-N}
N := 
\begin{bmatrix}
	\sum\limits_{j=1}^m g_{1,j} \|a_j\|_2^2 & -g_{1,2} \langle a_1, a_2 \rangle & \cdots & -g_{1,m} \langle a_1, a_m \rangle \\
	-g_{2,1} \langle a_2, a_1 \rangle & \sum\limits_{j=1}^m g_{2,j} \|a_j\|_2^2 & \cdots & -g_{2,m} \langle a_2, a_m \rangle \\
	\vdots & \vdots & \ddots & \vdots \\
	-g_{m,1} \langle a_m, a_1 \rangle & -g_{m,2} \langle a_m, a_2 \rangle & \cdots & \sum\limits_{j=1}^m g_{m,j} \|a_j\|_2^2
\end{bmatrix}.
\end{equation}

We now state a key property of the matrices {  $S_M$} and \( N \) defined above. All proofs in this section are deferred to Appendix~\ref{appendix-P}.

\begin{prop}\label{propMN}
	Suppose that \( \operatorname{rank}(A) \geq 2 \), and let {  $S_M$} and \( N \) be defined as in \eqref{sym-skew-decomp} and \eqref{definition-N}, respectively. Then both {  $S_M$} and \( N \) are positive definite. 
\end{prop}

We present the convergence result for Algorithm~\ref{alg:IRDR}.

\begin{theorem}\label{main-IRDR}
	Suppose that  the linear system $Ax=b$ is consistent, $\alpha\in(0,1)$, $\operatorname{rank}\left(A\right)\geq2$, and $x^0\in\mathbb{R}^n$ is an arbitrary initial vector.
	Let $x_*^0=A^{\dagger}b+(I-A^\dagger A)x^0$.
	Then the iteration sequence $\{x^k\}_{k\geq0}$ generated by Algorithm \ref{alg:IRDR} using Strategy I (sampling without replacement) satisfies
	$$\mathbb{E} [ \|x^k-x_*^0\|^2_2]
	\leq
	\left(1-4\alpha(1-\alpha)\frac{\sigma_{\min}^2({ S_M}^{\frac{1}{2}}A)}{\|A\|^2_F}
	\right)^k\|x^0-x_*^0\|^2_2,
	$$
	and the iteration sequence $\{x^k\}_{k\geq0}$ generated by Algorithm \ref{alg:IRDR} using Strategy II (volume sampling) satisfies
	$$
	\mathop{\mathbb{E}}\big[\|x^{k}-x_*^0\|^2_2] 
	\leq
	\left(
	1-8\alpha(1-\alpha)
	\frac{\sigma_{\min}^2(N^{\frac{1}{2}}A)} 
	{ \Vert A \Vert_F^4 -\Vert AA^{\top} \Vert_F^2} 
	\right)^{k}
	\|x^0-x_*^0\|^2_2,
	$$
	where ${ S_M}$ and $N$ are given by \eqref{sym-skew-decomp} and \eqref{definition-N}, respectively.
\end{theorem}

	We now compare the convergence upper bound of the IRDR and RDR methods.
	First we compare the single-step descent amount of Algorithm \ref{alg:IRDR} (with strategy I) with that of the basic RDR method. Actually,
	with $T_{\mathcal H_i}$ denoting the linear reflection operator defined in Appendix~\ref{appendix-P}, \eqref{reflection operator}, methods of the RDR type possess the following property:
	\begin{equation*}
		\begin{aligned}
			&\mathop{\mathbb{E}}\big[\|x^{k+1}-x_*^0\|^2_2 \mid x^k\big]=
			\mathop{\mathbb{E}}\big[\|(1-\alpha)x^{k}+\alpha z^k-x_*^0\|^2_2|x^k\big]\\
			=&\left(\alpha^2+(1-\alpha)^2\right)\|x^k-x_*^0\|^2_2
			+2\alpha(1-\alpha)
			\left\langle \mathop{\mathbb{E}}\left[ T_{\mathcal{H}_{i_{k_2}}}T_{\mathcal{H}_{i_{k_1}}}\right](x^k-x_*^0),x^k-x_*^0\right\rangle ,
		\end{aligned}
	\end{equation*}
	If indices \( i, j \in [m] \) are selected according to the basic RDR method \cite{han2024randomized}, then
	\[
	\mathbb{E}\left[ T_{\mathcal{H}_j} T_{\mathcal{H}_i} \right] = I - \frac{4}{\|A\|_F^2} A^\top \left(I- \frac{AA^\top}{\|A\|_F^2}\right) A.
	\]
	Fixing the iterative point $x^k$, let $\bar{x}^{k+1}_\text{I}$ denote the iterative point generated by Algorithm \ref{alg:IRDR} with strategy I,
	and $\tilde{x}^{k+1}$ denote the iterative point generated by basic RDR method, then we have
	\begin{equation}\label{IRDR-RDR_I}
		\begin{aligned}
			&\mathop{\mathbb{E}}\big[\|\tilde{x}^{k+1}-x_*^0\|^2_2 \mid x^k\big]
			-\mathop{\mathbb{E}}\big[\|\bar{x}^{k+1}_\text{I}-x_*^0\|^2_2 \mid x^k\big]\\
			=& \frac{4\alpha(1-\alpha)}{\|A\|_F^2} \left(A(x^k-x_*^0)\right)^\top \left(M-2\left(I- \frac{AA^\top}{\|A\|_F^2}\right)\right) \left(A(x^k-x_*^0)\right) .
		\end{aligned}
	\end{equation}
	The following proposition demonstrates that the single-step descent amount of Algorithm \ref{alg:IRDR} (with strategy I) is greater than that of the basic RDR method, i.e. the result in \eqref{IRDR-RDR_I} is positive.
	
	\begin{prop}\label{IRDR-SigleStep}
		Suppose that $\operatorname{rank}(A)\geq 2$, and let $M$ be defined as in \eqref{definition-M}.
		Then for any nonzero vector $x\in\mathbb{R}^m$, we have
		$$
		x^\top \left(M-2\left(I- \frac{AA^\top}{\|A\|_F^2}\right)\right) x > 0.
		$$
	\end{prop}
	
	Next we compare the convergence upper bounds between Algorithm \ref{alg:IRDR} (with strategy II) and the basic RDR method.
	Due to the structural specificity of the matrix $N$, analyzing the convergence bound for general $A$ is non-trivial.
	In what follows, we focus on such special form of $A$ to further compare the convergence upper bounds. Suppose that \( A \) is row-normalized and, in addition, the rows are orthogonal, i.e., \( \langle a_i, a_j \rangle = 0 \) for any \( i, j \in [m] \) with \( i \neq j \). In this case, the matrix \( N \) becomes
\[
N = (m - 1) I,
\]
and we have \( \|A\|_F^4 = m^2 \), \( \|AA^\top\|_F^2 = m \), and \( \|A\|_F^2 = m \). Under this assumption and with relaxation parameter \( \alpha = \frac{1}{2} \), Theorem~\ref{main-IRDR} implies that Algorithm~\ref{alg:IRDR} with strategy II satisfies
\[
\mathbb{E} \left[ \|x^k - x_*^0\|_2^2 \right]
\leq
\left(1 - 2\frac{\sigma_{\min}^2(A)}{\|A\|_F^2} \right)^k \|x^0 - x_*^0\|_2^2.
\]
Comparing this bound to that in \eqref{xie-RDR-rate}, we conclude that Algorithm~\ref{alg:IRDR} with strategy II also achieves a tighter convergence upper bound under the specified { condition}.

	\section{Acceleration by adaptive heavy-ball momentum}
\label{section4}	

In this section, building on the momentum variant of the RDR (mRDR) method \cite{han2024randomized}, we address its key limitations by developing an adaptive scheme that automatically tunes both the relaxation parameter and the momentum coefficient. At the \(k\)-th iteration (\(k \geq 1\)), the original mRDR method updates according to
\[
x^{k+1} = \left((1 - \alpha) I + \alpha \mathcal{R}_{\mathcal{H}_{i_{k_2}}} \mathcal{R}_{\mathcal{H}_{i_{k_1}}} \right)(x^k) + \beta(x^k - x^{k-1}),
\]
where \(\alpha\) and \(\beta\) are fixed hyperparameters. While this scheme can achieve linear convergence under suitable choices of \(\alpha\) and \(\beta\) \cite[Theorem 4.1]{han2024randomized}, its performance is highly sensitive to these parameters, which are often problem-specific and require manual tuning.

It follows from Theorem~\ref{main-IRDR} that Algorithm~\ref{alg:IRDR} converges linearly in expectation to the solution  
$
x_*^0 = A^{\dagger}b + (I - A^\dagger A)x^0.
$
Based on this result, we aim to choose the parameters \(\alpha_k\) and \(\beta_k\) at each iteration so as to minimize the error \(\|x^{k+1} - x_*^0\|_2\). This leads to the following constrained optimization problem:
\begin{equation}\label{AS-pro1}
	\begin{aligned}
		&\min_{\alpha, \beta \in \mathbb{R}} \quad  \|x - x_*^0\|_2^2 \\
		\text{subject to} \quad & x = \left((1 - \alpha) I + \alpha \mathcal{R}_{\mathcal{H}_{i_{k_2}}} \mathcal{R}_{\mathcal{H}_{i_{k_1}}} \right)(x^k) + \beta(x^k - x^{k-1})\\
		&\ \ =x^k-\alpha(x^k-z^k)+ \beta(x^k - x^{k-1}),
	\end{aligned}
\end{equation}
where $z^k= \mathcal{R}_{\mathcal{H}_{i_{k_2}}} \mathcal{R}_{\mathcal{H}_{i_{k_1}}} (x^k) $. 
Define
$$	
u_k := \frac{\langle a_{i_{k_1}}, x^k \rangle - b_{i_{k_1}}}{\|a_{i_{k_1}}\|_2^2} \ \ \text{and} \ \	v_k := \frac{\langle a_{i_{k_2}}, x^k \rangle - b_{i_{k_2}} - 2 \langle a_{i_{k_2}}, a_{i_{k_1}} \rangle u_k}{\|a_{i_{k_2}}\|_2^2}.
$$
Then we have 
$$z^k = x^k - 2 u_ka_{i_{k_1}} - 2 v_ka_{i_{k_2}}.$$
Hence, \eqref{AS-pro1} can be rewritten as
\begin{equation}\label{AS-pro2}
	\begin{aligned}
		\min\limits_{ \alpha,\beta\in\mathbb{R}}& \ \ \|x-x_*^0\|_2^2\\
		\text{subject to} \ \ x = x^k - & 2\alpha(u_k a_{i_{k_1}} + v_k a_{i_{k_2}} )  
		+ \beta (x^k - x^{k-1}).
	\end{aligned}
\end{equation}
If $\|x^k - x^{k-1}\|_2^2  \|u_k a_{i_{k_1}} + v_k a_{i_{k_2}}\|_2^2 -  \langle u_k a_{i_{k_1}} + v_k a_{i_{k_2}}, x^k - x^{k-1} \rangle^2 \neq 0$, then the minimizers of \eqref{AS-pro2} are given by
\begin{equation}\label{Parameter1}
	\left\{\begin{array}{ll}
		\alpha_k
		=
		\frac{ \|x^k - x^{k-1}\|_2^2 \left( u_k \langle x^k - x_*^0, a_{i_{k_1}} \rangle + v_k \langle x^k - x_*^0, a_{i_{k_2}} \rangle\right) -  \langle x^k - x_*^0, x^k - x^{k-1} \rangle \langle u_k a_{i_{k_1}} + v_k a_{i_{k_2}}, x^k - x^{k-1} \rangle} 
		{2\left( \|x^k - x^{k-1}\|_2^2  \|u_k a_{i_{k_1}} + v_k a_{i_{k_2}}\|_2^2 -  \langle u_k a_{i_{k_1}} + v_k a_{i_{k_2}}, x^k - x^{k-1} \rangle^2 \right)},
		\\
		\beta_k
		=
		\frac{\langle u_k a_{i_{k_1}} + v_k a_{i_{k_2}}, x^k - x^{k-1} \rangle \left( u_k \langle x^k - x_*^0,  a_{i_{k_1}}\rangle + v_k \langle x^k - x_*^0,  a_{i_{k_2}}\rangle \right)
			-   
			\langle x^k - x_*^0, x^k - x^{k-1} \rangle \|u_k a_{i_{k_1}} + v_k a_{i_{k_2}}\|_2^2} 
		{ \|x^k - x^{k-1}\|_2^2  \|u_k a_{i_{k_1}} + v_k a_{i_{k_2}}\|_2^2 -  \langle u_k a_{i_{k_1}} + v_k a_{i_{k_2}}, x^k - x^{k-1} \rangle^2}, 
	\end{array}
	\right.
\end{equation}
which seem to be intractable since $  \langle x^k - x_*^0,  a_{i_{k_1}}\rangle$, $ \langle x^k - x_*^0,  a_{i_{k_2}}\rangle$ and $\langle x^k-x_*^0 ,x^k - x^{k-1} \rangle$ include an unknown vector $x_*^0$.  In fact, we can calculate the three terms directly to get rid of $x_*^0$.
Noting that $AA^\dagger b = b$ and $x_*^0=A^\dagger b+(I-A^\dagger A)x^0$, so we have $Ax_*^0 = b$.
Thus, for every $i\in [m]$,\  $\langle a_i,x_*^0 \rangle = a_i^{\top}x_*^0 =  b_i$. Then we know that
\begin{equation}\label{kill-unknown}
	\langle x^k - x_*^0,  a_{i}\rangle 
	= \langle x^k,  a_{i}\rangle - \langle x_*^0,a_{i} \rangle
	= \langle x^k,  a_{i}\rangle - b_{i}.
	\nonumber
\end{equation}
On the other hand, we can discern from the problem \eqref{AS-pro2} that $x^k$ is the orthogonal projection of $x_*^0$ onto the affine set
$$
\begin{aligned}
	\Pi_{k-1}&:=x^{k-1}+\mbox{Span}\{ u_{k-1} a_{i_{(k-1)_1}} + v_{k-1} a_{i_{(k-1)_2}}, x^{k-1} - x^{k-2} \}\\
	&=x^{k-1}+\operatorname{Span}\{ x^{k-1}-z^{k-1}, x^{k-1} - x^{k-2} \},
\end{aligned}
$$
hence we have
$$\langle x^k-x_*^0 , x^k - x^{k-1} \rangle=0.$$
A geometric interpretation is presented in Figure~\ref{GI1}. Therefore,  \eqref{Parameter1} can be simplified to
\begin{equation}\label{Parameter2}
	\left\{\begin{array}{ll}
		\alpha_k
		=
		\frac{ \|x^k - x^{k-1}\|_2^2 \left( \langle  u_k a_{i_{k_1}} + v_k a_{i_{k_2}} , x^k \rangle  -  \left( u_k b_{i_{k_1}}+ v_k b_{i_{k_2}}\right) \right) } 
		{2\left( \|x^k - x^{k-1}\|_2^2  \|u_k a_{i_{k_1}} + v_k a_{i_{k_2}}\|_2^2 -  \langle u_k a_{i_{k_1}}+ v_k a_{i_{k_2}}, x^k - x^{k-1} \rangle^2 \right)}, 
		\\
		\beta_k
		=
		\frac{\langle u_k a_{i_{k_1}} + v_k a_{i_{k_2}}, x^k - x^{k-1} \rangle\left( \langle  u_k a_{i_{k_1}} + v_k a_{i_{k_2}} , x^k \rangle  -  \left( u_k b_{i_{k_1}} + v_k b_{i_{k_2}} \right) \right) } 
		{ \|x^k - x^{k-1}\|_2^2  \|u_k a_{i_{k_1}} + v_k a_{i_{k_2}}\|_2^2 -  \langle u_k a_{i_{k_1}} + v_k a_{i_{k_2}}, x^k - x^{k-1} \rangle^2}. 
	\end{array}
	\right.
\end{equation}
	Based on the definition of $z^k$, we have the relation $u_k a_{i_{k_1}} + v_k a_{i_{k_2}} = \frac{x^k - z^k}{2}$. To further simplify the notation, let 
	$$
	\begin{aligned}
		c_k
	:=&u_k b_{i_{k_1}} + v_k b_{i_{k_2}}\\
	= &\dfrac{\langle a_{i_{k_1}}, x^k \rangle b_{i_{k_1}}}{\|a_{i_{k_1}}\|_2^2}
	+ \dfrac{\langle a_{i_{k_2}}, x^k \rangle b_{i_{k_2}}}{\|a_{i_{k_2}}\|_2^2}
	- \left(
	\dfrac{b_{i_{k_1}}^2}{\|a_{i_{k_1}}\|_2^2}
	+ \dfrac{b_{i_{k_2}}^2}{\|a_{i_{k_2}}\|_2^2}
	\right) \\
	&- \dfrac{2\langle a_{i_{k_2}}, a_{i_{k_1}}\rangle
		(\langle a_{i_{k_1}}, x^k \rangle - b_{i_{k_1}}) b_{i_{k_2}}}
	{\|a_{i_{k_1}}\|_2^2 \|a_{i_{k_2}}\|_2^2}.
	\end{aligned}
	$$ 
	Then we have
	\begin{equation}\label{Parameter3}
			\left\{\begin{array}{ll}
			\alpha_k
			=\dfrac{\|x^k - x^{k-1}\|_2^2
				\left( \langle x^k - z^k, x^k \rangle -2 c_k \right)}
			{\|x^k - x^{k-1}\|_2^2 \|x^k - z^k\|_2^2
				- \langle x^k - z^k, x^k - x^{k-1} \rangle^2}, 
			\\
			\beta_k
			=
		\dfrac{\langle x^k - z^k, x^k - x^{k-1} \rangle
			\left(   \langle x^k - z^k, x^k \rangle - 2c_k \right)}
		{\|x^k - x^{k-1}\|_2^2 \|x^k - z^k\|_2^2
			- \langle x^k - z^k, x^k - x^{k-1} \rangle^2}. 
		\end{array}
		\right.
	\end{equation}
To sum up, for any $k\geq1$, if 
	$\|x^k - x^{k-1}\|_2^2  \|x^k - z^k\|_2^2 -  \langle x^k - z^k, x^k - x^{k-1} \rangle^2\neq 0$,
i.e. $\text{dim}(\Pi_k)=2$, can be guaranteed, then the minimizers of \eqref{AS-pro2} can be computed by \eqref{Parameter3}.

\begin{figure}[hptb]
	\centering
	\begin{tikzpicture}
		\draw (0,0)--(5,0)--(6,2)--(1,2)--(0,0);
		
		\filldraw (1.5,0.7) circle [radius=1pt]
		(4,3.5) circle [radius=1pt]
		(4,1.2) circle [radius=1pt];
		\draw (1.5,0.4) node {$x^{k-1}$};
		\draw (4.5,3.5) node {$x_*^0$};
		\draw (4.5,1.2) node {$x^k$};
		\draw (4.5,0.3) node {$\Pi_{k-1}$};
		\draw [dashed,-stealth] (4,3.5) -- (4,1.25);
		\draw [-stealth] (1.5,0.7) -- (3.95,1.19);
	\end{tikzpicture}
	\caption{A geometric interpretation of our design. The next iterate $x^k$  arises such that $x^k$ is the orthogonal projection of $x_*^0$ onto the affine set $\Pi_{k-1}=x^{k-1}+\mbox{Span}\{x^{k-1}-z^{k-1}, x^{k-1}-x^{k-2}\}$.}	
	\label{GI1}
\end{figure}
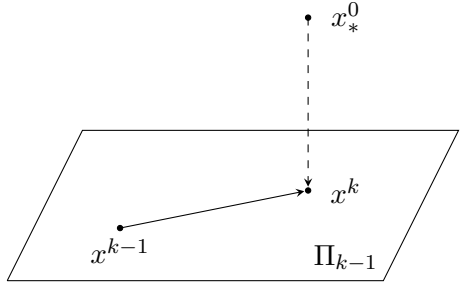

The following proposition indicates that if the hyperplane $\mathcal{H}_{i_{k_1}}$ and $\mathcal{H}_{i_{k_2}}$ are chosen such $z^k\neq x^k$ for $k\geq0$, then 
{ 
	$\|x^k - x^{k-1}\|_2^2  \|x^k - z^k\|_2^2 -  \langle x^k - z^k, x^k - x^{k-1} \rangle^2\neq 0$,
}
i.e. $\text{dim}(\Pi_k)=2$.
\begin{prop}\label{prob-full}
	Let \(x^0=x^{-1} \) be an initial point, and let \(\{x^k\}_{k \geq 1}\) be the sequence obtained by solving the optimization problem \eqref{AS-pro2}. If $z^i\neq x^i$ for \(i = 0, \ldots, k\), then $\text{dim}(\Pi_k)=2$.
\end{prop}

The following result indicates that if $\tilde{x}$ is not a solution to the linear system \(Ax = b\), then there exists an ordered pair \((i,j)\in[m]^2\) such that $\mathcal{R}_{\mathcal{H}_{j}} \mathcal{R}_{\mathcal{H}_{i}} (\tilde{x})\neq \tilde{x}$.
\begin{prop}\label{xie-0604-p}
	Suppose that \(\operatorname{rank}(A) \geq 2\), and for any ordered pair \((i,j)\in[m]^2\), define $z_{i,j}=\mathcal{R}_{\mathcal{H}_{j}} \mathcal{R}_{\mathcal{H}_{i}} (\tilde{x})$. If \(z_{i,j} = \tilde{x}\) holds for all \((i,j)\in[m]^2\), then \(\tilde{x}\) is a solution to the linear system \(Ax = b\).
\end{prop}

Now we have already constructed the {AmRDR} method described in Algorithm \ref{alg:AmRDR}. From Proposition \ref{xie-0604-p} we know that if  \(x^k\) is not a solution to the linear system, there exists $(i_{k_1},i_{k_2})$ such that \(z^k\neq x^k\), indicating that our requirement on the indices is well-defined.
	\begin{algorithm}[htpb]
		\caption{RDR with adaptive momentum (AmRDR) }
		\label{alg:AmRDR}
		\begin{algorithmic}[1]
        \Require $A\in \mathbb{R}^{m\times n}$, $b\in \mathbb{R}^m$, $k=1$, and initial vector $x^0\in \mathbb{R}^{n}$.
\Ensure The approximate solution $ x^k $.
			
				\item[1:] Select $(i_{0_1},i_{0_2})$ according to Strategy I (sampling without replacement) or Strategy II (volume sampling) and compute 
				$$z^0= \mathcal{R}_{\mathcal{H}_{i_{0_2}}} \mathcal{R}_{\mathcal{H}_{i_{0_1}}} (x^0). $$
				
				\item[2:] If $z^0=x^0$, return to Step 1. Otherwise, set $x^1=\frac{1}{2} x^0+\frac{1}{2} z^0$. 
			
				\item[3:]  Select $(i_{k_1},i_{k_2})$ according to Strategy I (sampling without replacement) or Strategy II (volume sampling) and 
				 compute $$z^k= \mathcal{R}_{\mathcal{H}_{i_{k_2}}} \mathcal{R}_{\mathcal{H}_{i_{k_1}}} (x^k). $$
				
				\item[4:] If $z^k=x^k$, return to Step 3. Otherwise, compute $\alpha_k$ and $\beta_k$ according to \eqref{Parameter3}.
				
				\item[5:] Update
				$$
				x^{k+1}=(1-\alpha_k) x^k+\alpha_k z^k+\beta_k(x^k-x^{k-1}).
				$$
				\item[6:] If the stopping rule is satisfied, stop and go to output. Otherwise, set $k=k+1$ and return to Step 3.

		\end{algorithmic}
	\end{algorithm}

	\begin{remark}The condition $z^i=x^i$ for $i\geq 0$ in Algorithm~\ref{alg:AmRDR} may trigger resampling. Proposition~\ref{xie-0604-p} guarantees that, whenever the current iterate is not a solution, there exists at least one ordered pair that yields a nontrivial reflection. However, it does not provide a uniform lower bound on the probability of sampling such a pair. In our implementations, we observed that resampling occurred rarely for the problem instances considered in this paper. Nevertheless, resampling may occur more frequently for problems with special structure. For example, if the current iterate already satisfies a large collection of hyperplanes that carry a substantial portion of the sampling probability, selecting two such hyperplanes leaves the iterate unchanged and therefore triggers resampling. Thus, although resampling was rarely encountered in our experiments, its frequency may increase for specially structured problems.
			\end{remark}

			\begin{remark}\label{remark2}
				Define
			$
			D_k
			:=
			\|x^k-x^{k-1}\|_2^2\|x^k-z^k\|_2^2
			-
			\langle x^k-z^k,x^k-x^{k-1}\rangle^2,
			$
			which is the Gram determinant associated with the two search directions
			$x^k-z^k$ and $x^k-x^{k-1}$.
			Once $z^k\neq x^k$ is accepted,
			Proposition~\ref{prob-full} guarantees
			$\dim(\Pi_k)=2$ in exact arithmetic. Consequently,
			$
			x^k\neq x^{k-1}$ and $
			D_k>0$,
			which implies that the two search directions are linearly independent and that
			the denominator in the formulas for $\alpha_k$ and $\beta_k$ is nonzero.
		In finite-precision computation, however, $D_k$ may become numerically very small.
		In such cases, we set $\alpha_k=\frac{1}{2}$ and $\beta_k=0$, so that the current
		update reduces to the standard RDR step.
	\end{remark}

\subsection{Convergence analysis}

We begin by introducing some auxiliary variables. Define
\[
\tilde{x}^{k+1} \coloneqq \frac{1}{2} x^k + \frac{1}{2} z^k.
\]
It is evident that \(\tilde{x}^{k+1}\) corresponds to the next iterate generated by Algorithm~\ref{alg:IRDR} with a fixed parameter \(\alpha = \frac{1}{2}\), when starting from \(x^k\). Moreover, it is clear that \(\tilde{x}^{k+1} \in \Pi_k\).
Next, we define the index set
\[
\mathcal{Q}_k := \left\{ (i, j) \mid z^k = \mathcal{R}_{\mathcal{H}_j} \mathcal{R}_{\mathcal{H}_i}(x^k) \neq x^k \right\},
\]
which collects all index pairs that yield a nontrivial reflection at iteration \(k\).
We also define the auxiliary vector
\[
\zeta_k := \langle z^k - x^k,\, x^k - x^{k-1} \rangle (z^k - x^k) - \|z^k - x^k\|_2^2 (x^k - x^{k-1}),
\]
and denote by \(\theta_k\) the angle between \(\tilde{x}^{k+1} - x_*^0\) and \(\zeta_k\).
Since $\mathcal{Q}_k$ is a finite set, we further define the scalar
\begin{equation}\label{gamma_k}
	\gamma_k \coloneqq \min_{(i, j) \in \mathcal{Q}_k} \left\{ \cos^2 \theta_k \right\}.
\end{equation}
Based on the above definitions, we now present the convergence result for Algorithm~\ref{alg:AmRDR}. It can be observed that the AmRDR method exhibits a convergence bound that is at least as tight as that of the IRDR method.

\begin{theorem}\label{main-AmRDR}
	Suppose that the linear system $Ax=b$ is consistent, $\operatorname{rank}(A)\geq2$, and $x^0\in\mathbb{R}^n $ is the arbitrary initial vector. Let $x_*^0=A^{\dagger}b+(I-A^\dagger A)x^0$. 
	Then the iteration sequence $\{x^k\}_{k\geq0}$ generated by Algorithm \ref{alg:AmRDR} using strategy I (sampling without replacement) satisfies
	$$
	\mathbb{E} \left[ \|x^{k+1} - x_*^0\|_2^2 \mid x^k \right] \leq (1 - \gamma_k) 
	\left(1-\frac{ \sigma_{\min}^2({ S_M}^{\frac{1}{2}}A)}{\|A\|^2_F}
	\right)
	\|x^k - x_*^0\|_2^2,
	$$
	and the iteration sequence $\{x^k\}_{k\geq0}$ generated by Algorithm \ref{alg:AmRDR} using strategy II  (volume sampling)  satisfies
	$$
	\mathbb{E} \left[ \|x^{k+1} - x_*^0\|_2^2 \mid x^k \right] \leq (1 - \gamma_k) 
	\left(
	1-2
	\frac{\sigma_{\min}^2(N^\frac{1}{2}A)}
	{ \Vert A \Vert_F^4 -\Vert AA^{\top} \Vert_F^2}
	\right)
	\|x^k - x_*^0\|_2^2,
	$$
	where $\gamma_k$ is given by \eqref{gamma_k}, and ${ S_M}$ and $N$ are given by \eqref{sym-skew-decomp} and \eqref{definition-N}, respectively.
\end{theorem}

	First, it follows from the definition of $\gamma_k$ that $\gamma_k\geq 0$. Hence, omitting the factor $1-\gamma_k$ in Theorem~\ref{main-AmRDR} recovers the convergence upper bound for IRDR with $\alpha=\frac{1}{2}$. 	Therefore, AmRDR converges globally and linearly in expectation, with a convergence bound that is at least as tight as that of IRDR.
	
Next, we provide a further geometric interpretation of the factor $\gamma_k$.
	Since $x^k$ is the orthogonal projection of $x_*^0$ onto $\Pi_{k-1}$, we have
$
	\langle x^k-x_*^0,x^k-x^{k-1}\rangle=0.
	$
	Moreover, since $x_*^0$ lies on both selected hyperplanes, the reflection property yields
	$
	\|z^k-x_*^0\|_2=\|x^k-x_*^0\|_2
	$
	and
	$
	\langle x^k-x_*^0,z^k-x^k\rangle
	=-\frac{1}{2}\|z^k-x^k\|_2^2.
	$
Combining these identities with the definition of $\zeta_k$ yields
\[
\left\langle \widetilde{x}^{k+1}-x_*^0,\zeta_k\right\rangle
=
-\frac{1}{2}\|z^k-x^k\|_2^2\cdot
\langle z^k-x^k,x^k-x^{k-1}\rangle.
\]
Therefore, whenever $\dim(\Pi_k)=2$,
$
\cos^2\theta_k>0
$
is equivalent to
$
\langle z^k-x^k,x^k-x^{k-1}\rangle\neq0.
$
This equivalence eliminates the unknown vector $x_*^0$ from the condition
for strict positivity of the angle factor.

It also suggests a simple strict-angle modification of
Algorithm~\ref{alg:AmRDR}. After computing $z^k$ and verifying that
$z^k\neq x^k$, one may additionally check whether
\[
\langle z^k-x^k,x^k-x^{k-1}\rangle\neq0.
\]
If this condition fails, a new pair is sampled. Under this modified rule,
the angle factor associated with every accepted pair is strictly positive. By Proposition~\ref{xie-0604-p}, if $x^k$ is not a solution, then $\mathcal{Q}_k$ is nonempty, and hence $\gamma_k>0$.
Consequently, under this modification, the $\gamma_k$-dependent estimate
in Theorem~\ref{main-AmRDR} provides a strict one-step improvement over
the baseline convergence bound for Algorithm~\ref{alg:AmRDR}.

Finally, we note that the practical benefits of the proposed methods depend on both the matrix structure and the computational setting. Strategy II incurs preprocessing and storage costs that can be substantial when $AA^\top$ is dense or the matrix is extremely large. For a single solve, or when sampling overhead dominates the iteration cost, these costs may outweigh the reduction in row actions, making Strategy I a more economical alternative. Reusing the same coefficient matrix across multiple right-hand sides allows the preprocessing cost to be amortized. For either sampling strategy, the adaptive subproblem may become ill-conditioned in finite precision when the two search directions are nearly linearly dependent. Proposition~\ref{prob-full} guarantees nondegeneracy for accepted iterations in exact arithmetic, but does not ensure numerical conditioning.

\section{Numerical { experiments}}

In this section, we implement the { IRDR} method (Algorithm \ref{alg:IRDR}) and the { AmRDR} method (Algorithm \ref{alg:AmRDR}). 
We denote the { IRDR} method with strategy I and  strategy II  as { IRDR}-I and { IRDR}-II, respectively. 
Similarly, we represent the { AmRDR} method with strategy I and  strategy II as { AmRDR}-I and { AmRDR}-II, respectively.
We also compare our methods with RK, RDR, and momentum variant of RDR (mRDR) \cite{han2024randomized}. 
All the methods are implemented in  {\sc Matlab} R2022a for Windows $11$ on a desktop PC with Intel(R) Core(TM) i7-1360P CPU @ 2.20GHz  and 32 GB memory. 
The code to reproduce our results can be found at  \href{https://github.com/xiejx-math/AmRDR}{https://github.com/xiejx-math/AmRDR}.

\subsection{When to use volume sampling}
\label{sect-when}
Although volume sampling is a powerful technique in the design of randomized algorithms, its computational cost can be significant. Nevertheless, a number of exact and approximate algorithms have been developed for efficiently generating random subsets \(\mathcal{S} \sim \operatorname{Vol}_s(AA^\top)\); see, for example, \cite{kulesza2012determinantal, anari2016monte, derezinski2024solving}. In our experiments, Strategy II uses the exact volume-sampling procedure
	for $s=2$ described in \cite[Section 5]{xiang2025randomized}.
	To characterize the preprocessing cost, we define
	\[
	\mathcal T:=\{(i,\ell):a_i\circ a_\ell\neq 0\},
	\qquad
	s_{i,\ell}:=\|a_i\circ a_\ell\|_0,
	\]
	where $\circ$ denotes the Hadamard product and $\|\cdot\|_0$ denotes the number of nonzero entries.
	The preprocessing stage computes the required entries of the sparse
	Gram matrix $AA^\top$ and constructs cumulative lookup tables for sampling.
	Once these tables are available, each pair is sampled using two binary
	searches and $O(1)$ additional arithmetic operations, resulting in an
	$O(\log m)$ cost per sample. 
  
For each $(i,\ell)\in\mathcal T$ with $i\le\ell$, computing
$(AA^\top)_{i\ell}=\langle a_i,a_\ell\rangle$ requires
$2s_{i,\ell}-1$ flops.
Thus, computing the required Gram-matrix entries costs
$
\sum_{\substack{(i,\ell)\in\mathcal T, i\le\ell}}
(2s_{i,\ell}-1)
$
flops, while constructing the remaining cumulative quantities requires
$O(|\mathcal T|+m)$ operations.
The overall preprocessing complexity is therefore
\[
O\!\left(
\sum_{\substack{(i,\ell)\in\mathcal T, i\le\ell}}
s_{i,\ell}+|\mathcal T|+m
\right).
\]
In the dense worst case, $|\mathcal T|=O(m^2)$ and
$s_{i,\ell}\le n$, yielding a preprocessing cost of $O(m^2n)$.
The storage requirement, by contrast, is
$O(\operatorname{nnz}(AA^\top)+|\mathcal T|+m)$
for the required Gram-matrix entries and auxiliary sampling structures,
which reduces to $O(m^2)$ in the dense case.
This preprocessing cost can be amortized over multiple linear systems
with the same coefficient matrix $A$ and different right-hand sides $b$.
Since the sampling distribution depends only on $A$, the preprocessing
is performed only once, and the resulting sampling tables can be cached
and reused for all subsequent systems.

\subsection{Numerical setup}

To investigate the effect of geometric relationships between hyperplanes on algorithm performance, we construct controlled test problems through the following generation procedure. 
Given parameters $m, n, r, \sigma_{1}$, and $\delta$, we construct matrices \(A = U D V^\top\), where \(U \in \mathbb{R}^{m \times r}\) and \(V \in \mathbb{R}^{n \times r}\) are column-orthogonal matrices. Using MATLAB notation, we generate the column-orthogonal matrices with the following commands: {\tt [U,$\sim$]=qr(randn(m,r),0)} and {\tt [V,$\sim$]=qr(randn(n,r),0)}.  The diagonal matrices \(D\) have entries
$
D_{1,1} =\sigma_{1}$  and $D_{i,i}=\delta, i=2,\ldots,r.
$
This construction allows the angular relationships between the hyperplanes to be adjusted by varying the ratio between \(\sigma_1\) and \(\delta\). A larger value of \(\sigma_1\) relative to \(\delta\) leads to more pronounced angular differences between the hyperplanes. 
In addition, we also consider a different approach, in which the entries of the matrix \(A\) are drawn independently from a uniform distribution over the interval \([t, 1]\). Varying the value of \(t\) affects the coherence of \(A\). As \(t\) approaches $1$, the rows of \(A\) become increasingly correlated, resulting in higher coherence. 

For both IRDR and mRDR, we set \(\alpha = 1/2\). For each test problem, the momentum parameter \(\beta\) for mRDR is selected from \(\{0, 0.05, 0.10, \ldots, 0.50\}\) to minimize the average number of iterations required to reach the prescribed tolerance for the relative solution error (RSE), defined as
$
\mathrm{RSE} :=
\frac{\|x^k - A^\dagger b\|_2^2}{\|x^0-A^\dagger b\|_2^2}.
$
The selected value of \(\beta\) is held fixed throughout each run. In fact, this empirical parameter-selection strategy follows the numerical approach of \cite{han2024randomized} and is intended to provide a favorable choice of \(\beta\) for the mRDR method. To ensure the consistency of the linear system, we first generate a vector \( x^* \) that follows a standard normal distribution, and then set \( b = Ax^* \). All computations are initialized with $x^0=0$. We consider $\|z^k-x^k\|_2$ as zero when it is less than {\tt eps}.  

We evaluate the performance of the methods using the average number of row actions and the elapsed wall-clock time over $20$ independent trials. Before collecting timing measurements, we perform a warm-up run of each method on a separate, independently generated linear system.  Both the warm-up runs and matrix generation are excluded from the reported timings. Elapsed wall-clock times are measured using {\sc Matlab}'s {\tt tic} and {\tt toc} functions.  Finally, we note that the elapsed times for IRDR-II and AmRDR-II include the volume-sampling preprocessing cost, whereas the parameter-tuning time for mRDR is excluded. Thus, the timing results for mRDR characterize its performance with a preselected momentum parameter.

\subsection{Comparison to the other methods: randomly generated instances}

In this subsection, we compare the performance of IRDR-I, IRDR-II, AmRDR-I, and AmRDR-II with the RK, RDR, and mRDR methods. 
Figures~\ref{figure1} and \ref{figure2} compare the methods on full-column-rank and rank-deficient test systems, respectively, with $m=500$, $n=100$, $\delta=1$, and $\sigma_1\in\{10,50,100\}$.
The bold line illustrates the median value derived from $20$ trials. The lightly shaded area signifies the  range from the minimum to the maximum values, while the darker shaded one indicates the data lying between the $25$-th and $75$-th quantiles.

\begin{figure}[tbhp]
	\centering
	\begin{tabular}{ccc} 
		\includegraphics[width=0.33\linewidth]{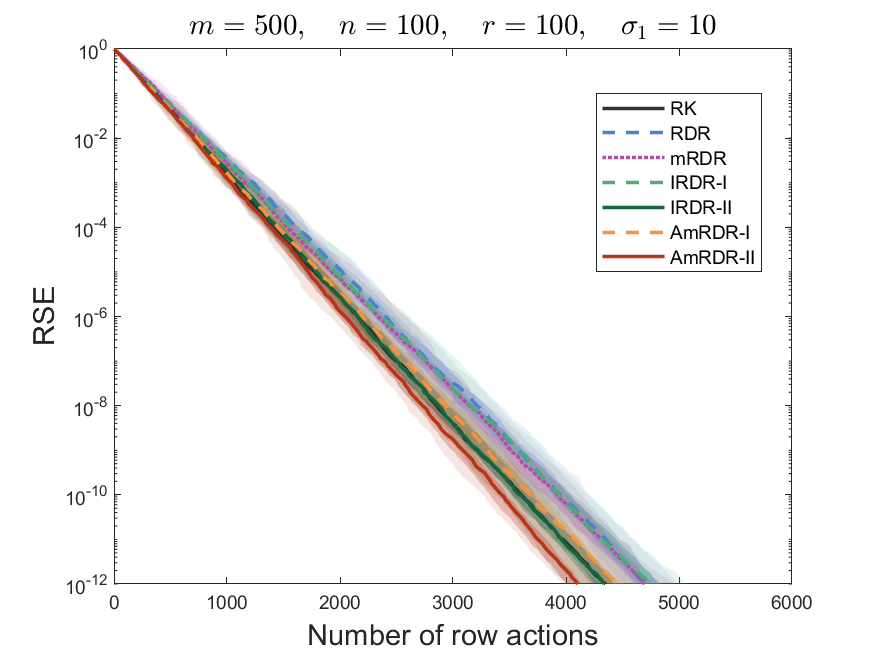}
		\includegraphics[width=0.33\linewidth]{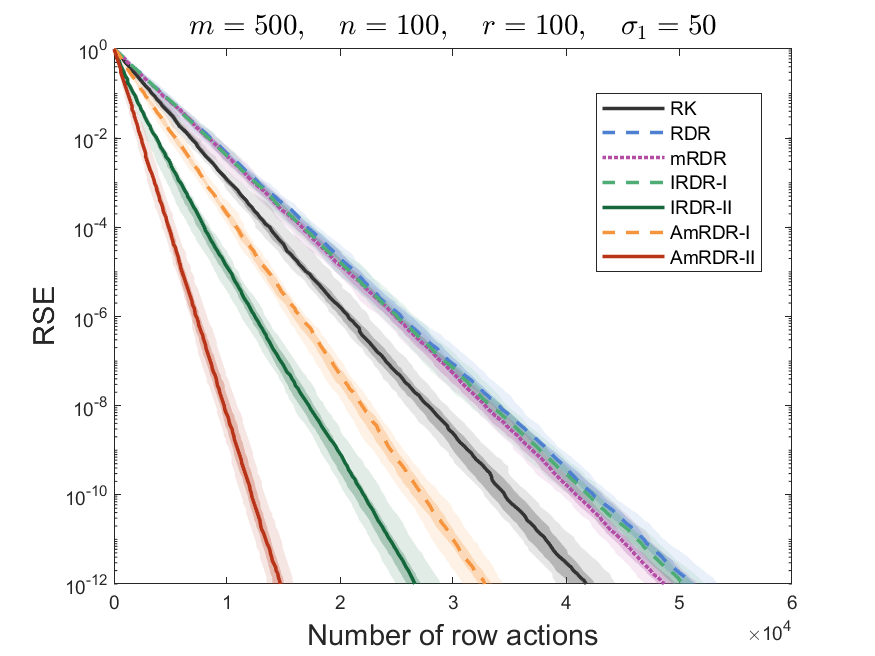}
		\includegraphics[width=0.33\linewidth]{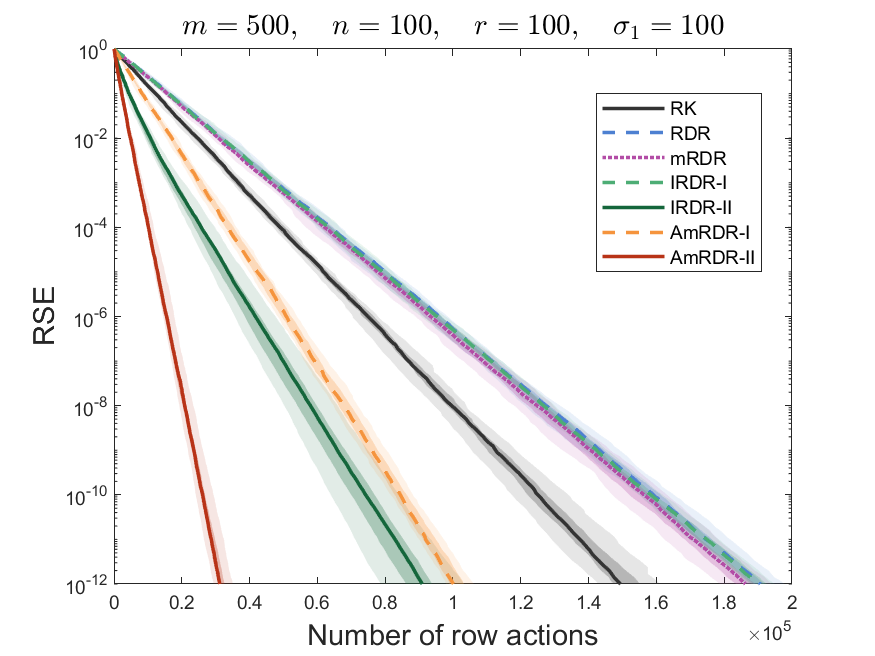}\\
		\includegraphics[width=0.33\linewidth]{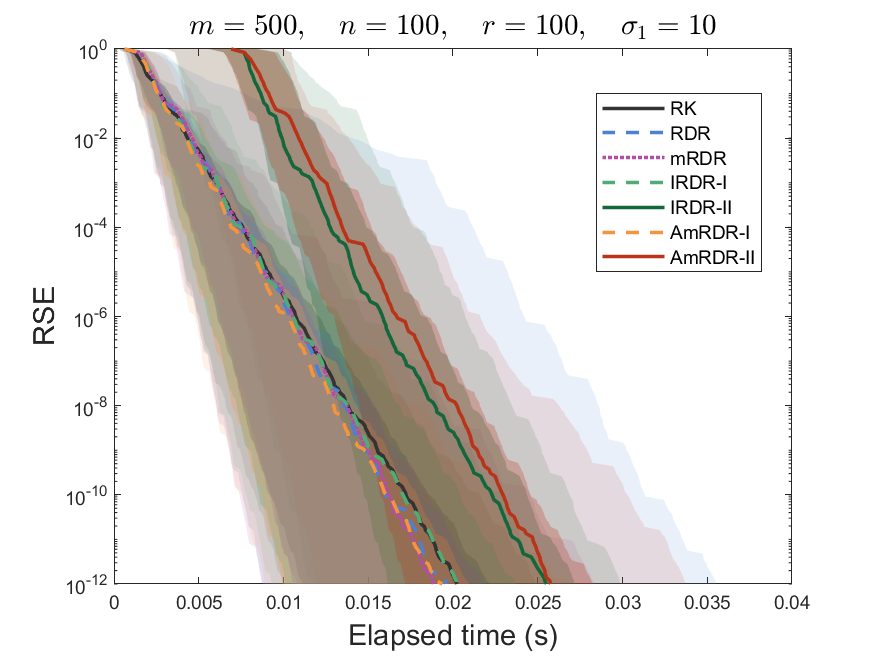}
		\includegraphics[width=0.33\linewidth]{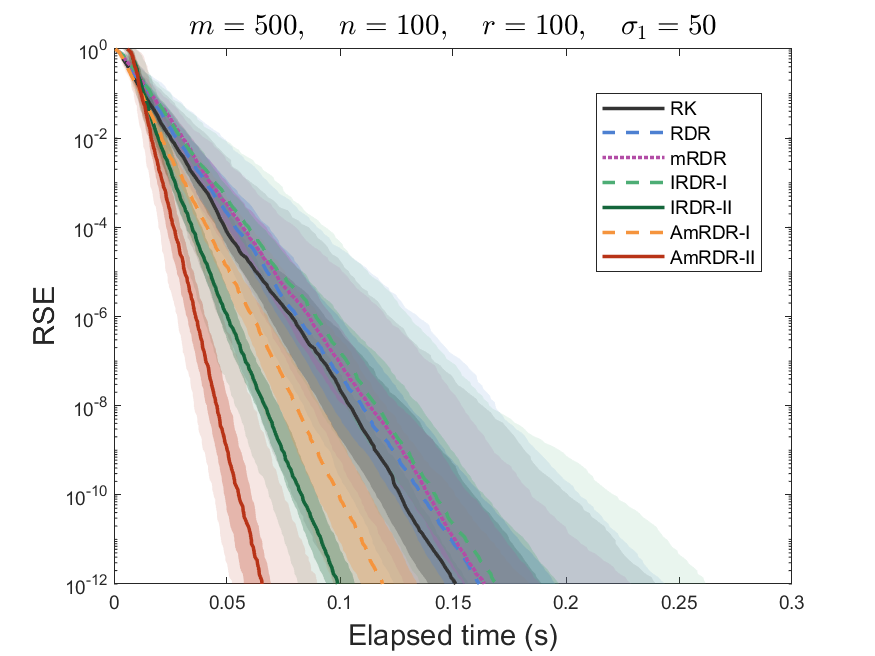}
		\includegraphics[width=0.33\linewidth]{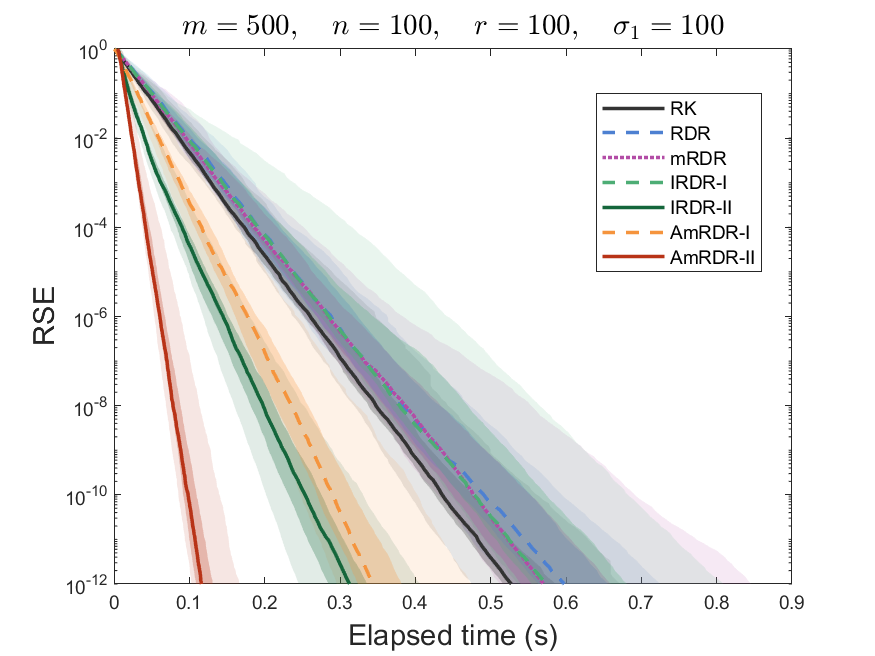}
	\end{tabular}
	\caption{Performance of RK, RDR, mRDR, { IRDR}-I, { IRDR}-II, { AmRDR}-I, and { AmRDR}-II for linear systems with full rank Gaussian matrix. Figures depict the number of row actions and the elapsed time (in seconds) vs RSE.  The title of each plot indicates the values of $m, n$, $r$, and $\sigma_1$. We fix $\delta=1$. }
	\label{figure1}
\end{figure}

\begin{figure}[tbhp]
	\centering
	\begin{tabular}{ccc} 
		\includegraphics[width=0.33\linewidth]{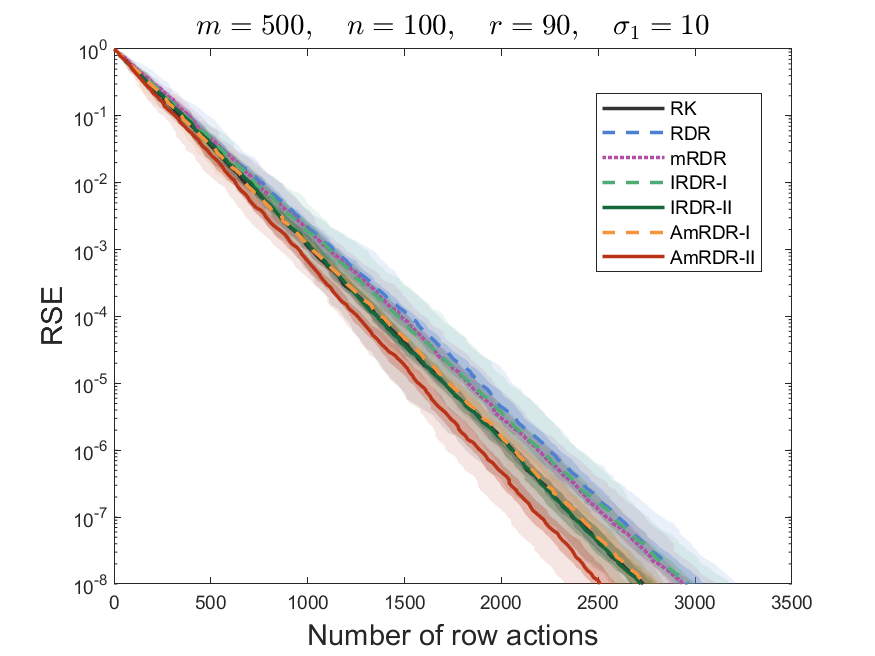}
		\includegraphics[width=0.33\linewidth]{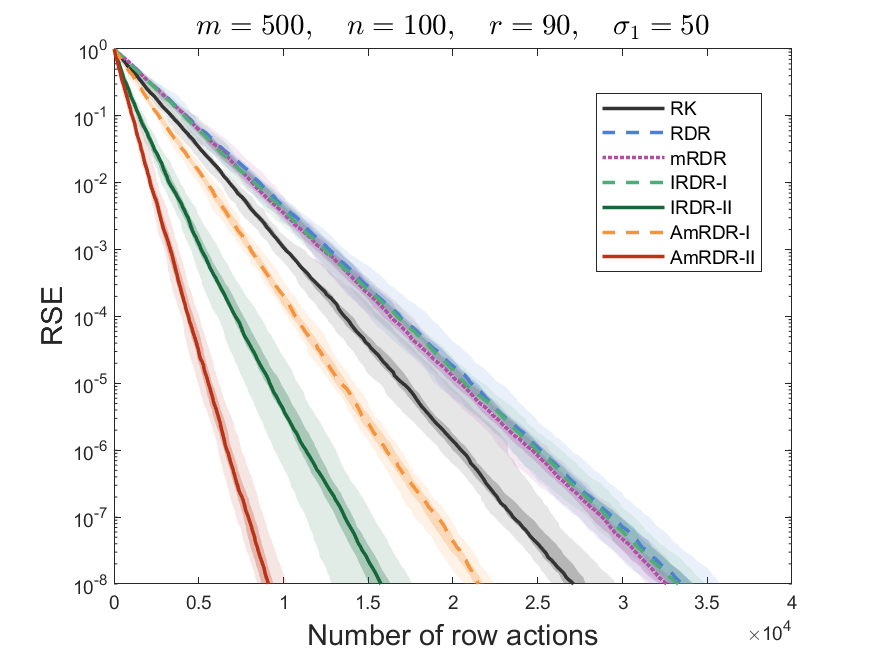}
		\includegraphics[width=0.33\linewidth]{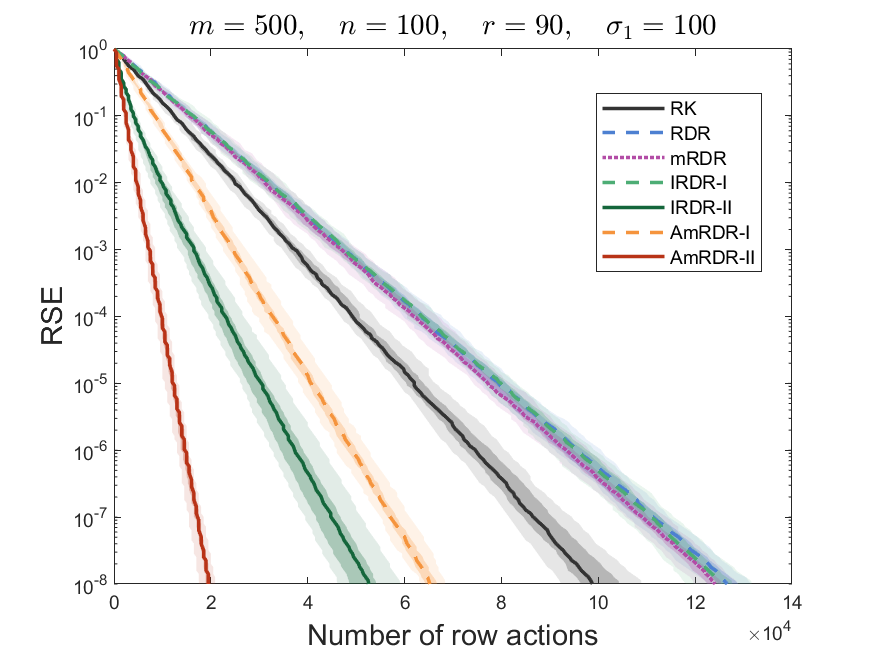}\\
		\includegraphics[width=0.33\linewidth]{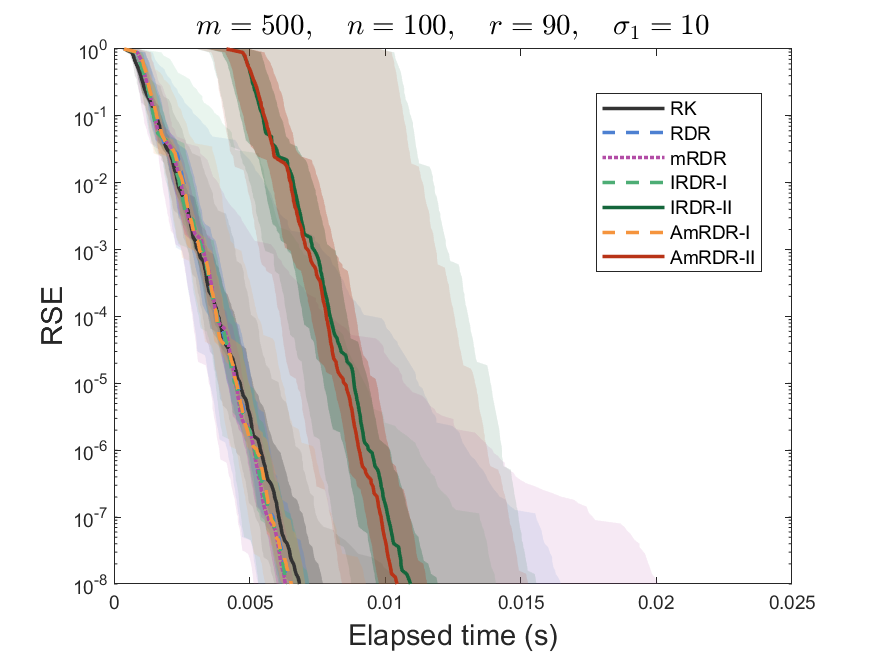}
		\includegraphics[width=0.33\linewidth]{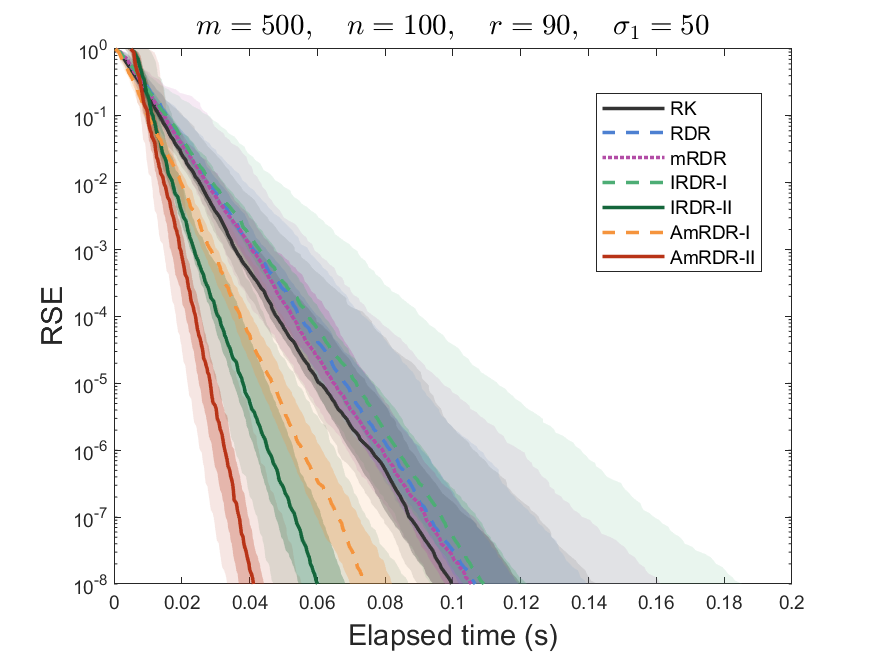}
		\includegraphics[width=0.33\linewidth]{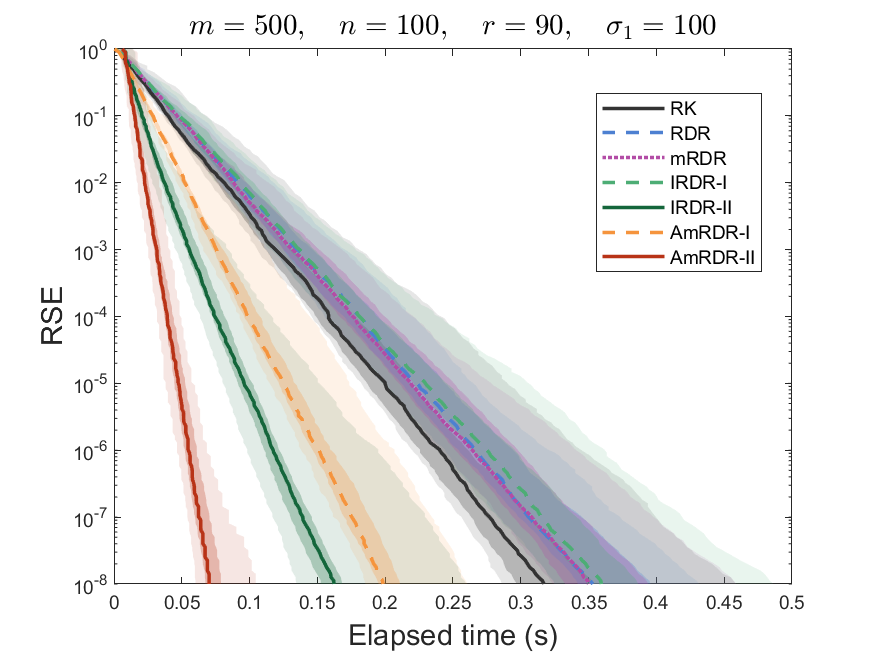}
	\end{tabular}
	\caption{Performance of RK, RDR, mRDR, { IRDR}-I, { IRDR}-II, { AmRDR}-I, and { AmRDR}-II for linear systems with rank-deficient Gaussian matrix. Figures depict the number of row actions and the elapsed time (in seconds) vs RSE.  The title of each plot indicates the values of $m, n$, $r$, and $\sigma_1$. We fix $\delta=1$. }
	\label{figure2}
\end{figure}

It can be observed from Figures~\ref{figure1} and \ref{figure2} that when $\sigma_1=10$, all methods require comparable numbers of row actions to attain a given RSE. As \(\sigma_1\) increases to \(50\) and \(100\), the differences become more pronounced in both the full-column-rank and rank-deficient cases. In particular, IRDR-II and AmRDR-II require fewer row actions than their Strategy-I counterparts, with AmRDR-II achieving the smallest number of row actions among the methods tested. The improvement of AmRDR-II over IRDR-II demonstrates the benefit of adaptive momentum combined with volume sampling in these experiments. By contrast, AmRDR-I does not consistently improve upon IRDR-I, indicating that the benefit of adaptive momentum depends on the sampling strategy and the test problem. The elapsed-time results highlight the tradeoff between the reduction in row actions and the additional cost of volume sampling. For \(\sigma_1=10\), the modest reduction in row actions does not compensate for this cost, and the Strategy-II methods take longer than the other methods. For \(\sigma_1=50\) and \(100\), their larger reductions in row actions lead to lower elapsed times, even with preprocessing included. AmRDR-II achieves the lowest average elapsed time in these cases. These results indicate that, for the tested systems, volume sampling becomes more beneficial as the dominant singular value increases, whereas Strategy I remains competitive when the gain from volume sampling is small.

To further evaluate the methods on larger systems, we compare  RK, RDR, mRDR, IRDR-I, IRDR-II, AmRDR-I, and AmRDR-II with \(m\in\{1000,5000,10000\}\), \(n=r=100\), \(\sigma_1=100\), and \(\delta=1\). To assess the benefit of reusing the volume-sampling preprocessing, we solve ten linear systems sharing each independently generated coefficient matrix \(A\), with independently generated right-hand sides \(b=Ax^*\). For IRDR-II and AmRDR-II, preprocessing is performed once per matrix and reused across the ten solves. Its cost is included in the average elapsed time per solve by allocating one tenth of the preprocessing time to each solve.  
Across all three matrix sizes, AmRDR-II requires the fewest row actions, followed by IRDR-II and AmRDR-I.  The reduction in row actions achieved by the Strategy-II methods persists as \(m\) increases, demonstrating the effectiveness of volume sampling for the tested systems. Moreover, AmRDR-II improves upon IRDR-II, illustrating the additional benefit of adaptive momentum.
The elapsed-time results show that the reduction in row actions translates into a computational advantage when the preprocessing is reused across multiple right-hand sides. Both Strategy-II methods achieve shorter average elapsed times than the other methods, with AmRDR-II being the fastest in all three cases. At \(m=10000\), the initial time offset associated with preprocessing becomes more noticeable, but the Strategy-II methods retain their advantage after this cost is distributed over the ten solves. These results support the use of volume sampling for repeated solves with a shared coefficient matrix, where its preprocessing cost can be amortized.

\begin{figure}[tbhp]
	\centering
	\begin{tabular}{ccc} 
		\includegraphics[width=0.33\linewidth]{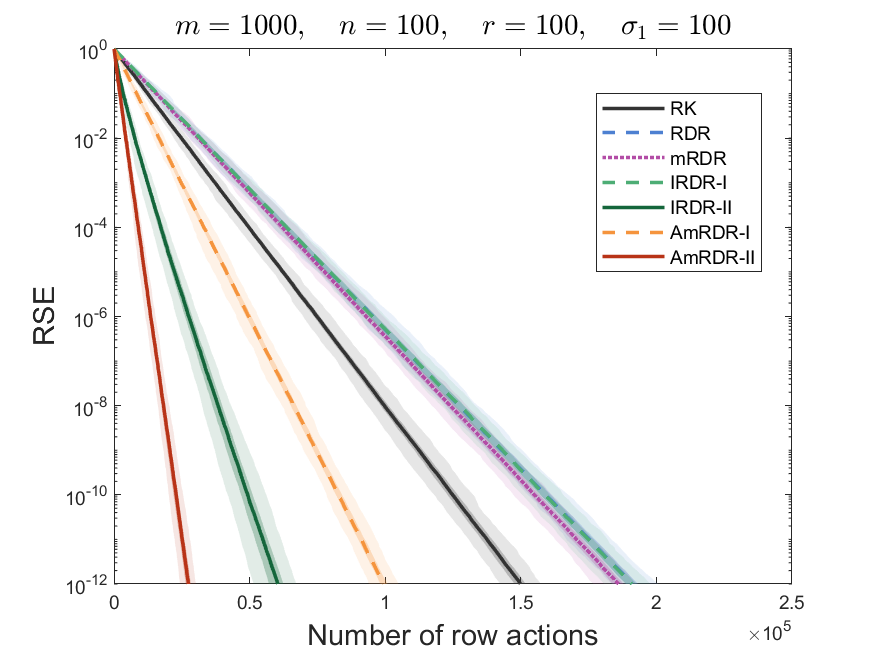}
		\includegraphics[width=0.33\linewidth]{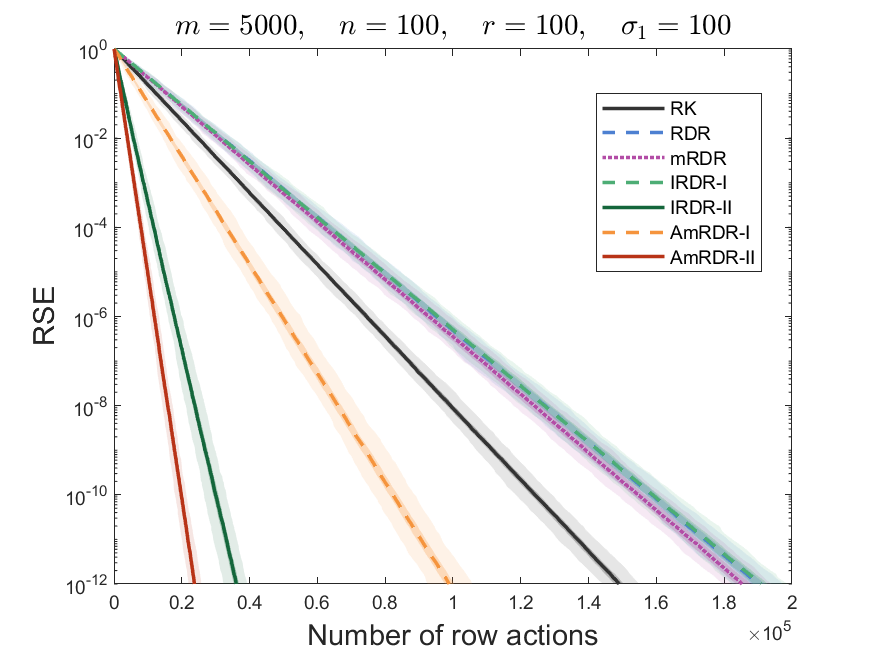}
		\includegraphics[width=0.33\linewidth]{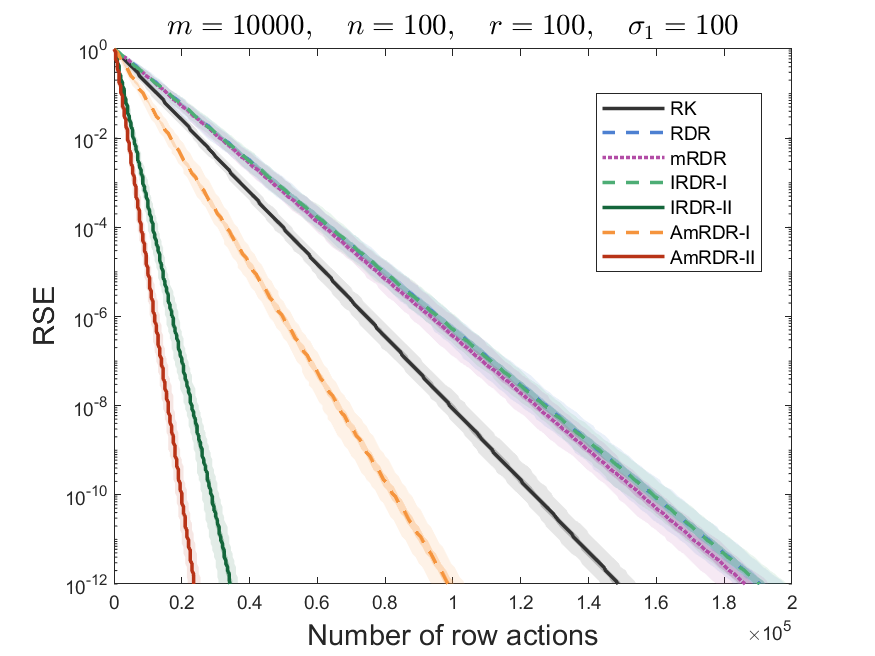}\\
		\includegraphics[width=0.33\linewidth]{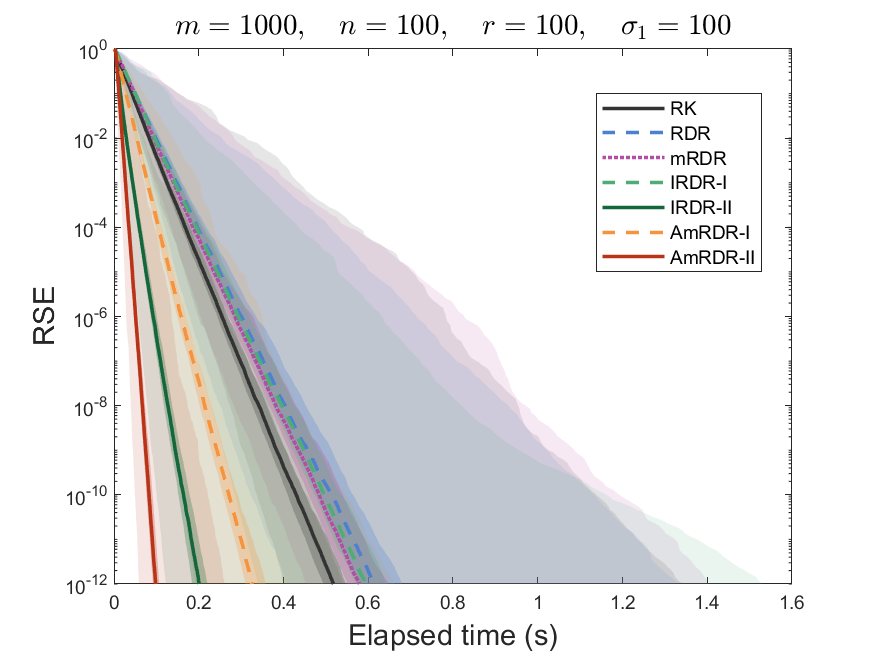}
		\includegraphics[width=0.33\linewidth]{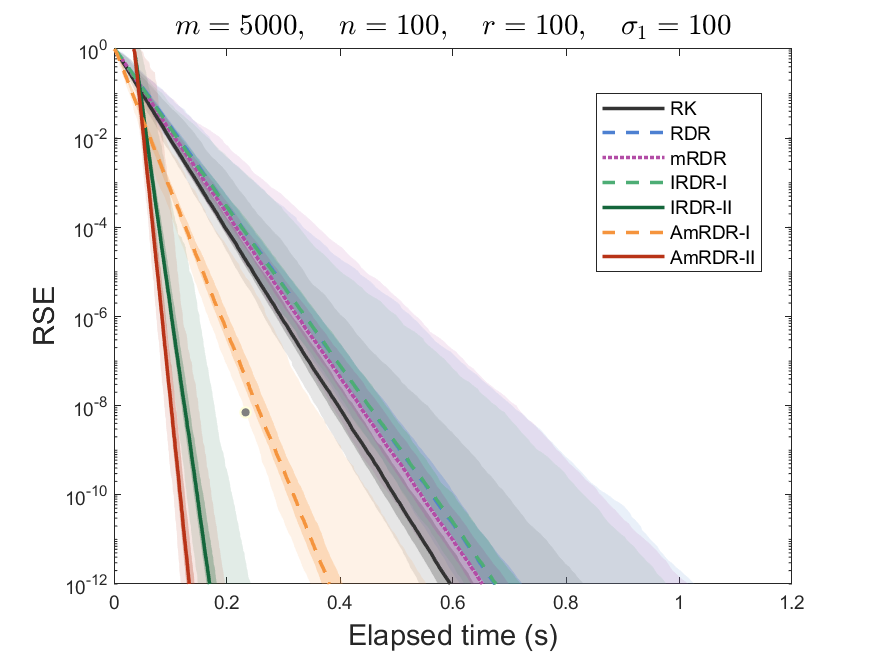}
		\includegraphics[width=0.33\linewidth]{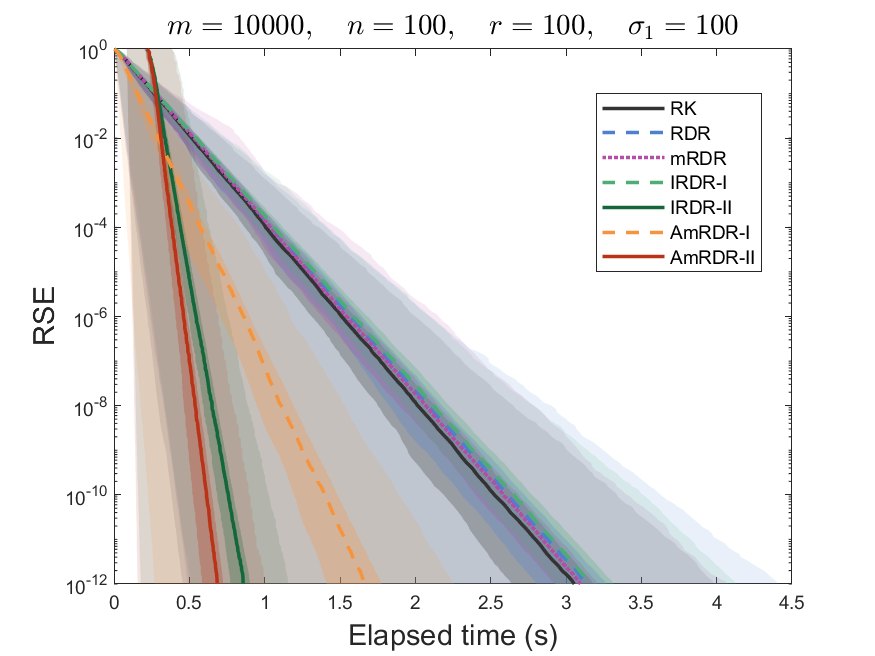}
	\end{tabular}
	\caption{Performance of RK, RDR, mRDR, IRDR-I, IRDR-II, AmRDR-I, and AmRDR-II
		{ for large-scale matrices.}
		Figures depict the number of row actions and the elapsed time (in seconds) vs RSE. The title of each plot indicates the values of $m$, $n$, $r$ and $\sigma_1$. We fix $\delta=1$.  
	}
	\label{NE4}
\end{figure}

Figure~\ref{figure3} presents the results for \(m=500\), \(n=100\), and \(t\in\{0.3,0.6,0.9\}\) for  the uniform distribution matrix.  As \(t\) increases, the rows become more strongly correlated, and all methods require substantially more row actions to attain the prescribed accuracy. Across the three values of \(t\), AmRDR-I and AmRDR-II require the fewest row actions, followed by mRDR. Their improvement over the corresponding non-momentum variants demonstrates the effectiveness of adaptive momentum for these test problems. In contrast, RK, RDR, IRDR-I, and IRDR-II exhibit relatively similar performance in terms of row actions. Moreover, the curves for AmRDR-I and AmRDR-II nearly overlap, indicating that volume sampling provides little additional reduction in row actions over sampling without replacement in this setting.
The elapsed-time results also favor the adaptive momentum methods, particularly for \(t=0.6\) and \(t=0.9\). Although AmRDR-I and AmRDR-II require comparable numbers of row actions, AmRDR-I achieves slightly shorter average elapsed times.   These results suggest that sampling without replacement combined with adaptive momentum offers a favorable balance between convergence and computational cost for the tested uniformly distributed matrices.

\begin{figure}[tbhp]
	\centering
	\begin{tabular}{ccc} 
		\includegraphics[width=0.33\linewidth]{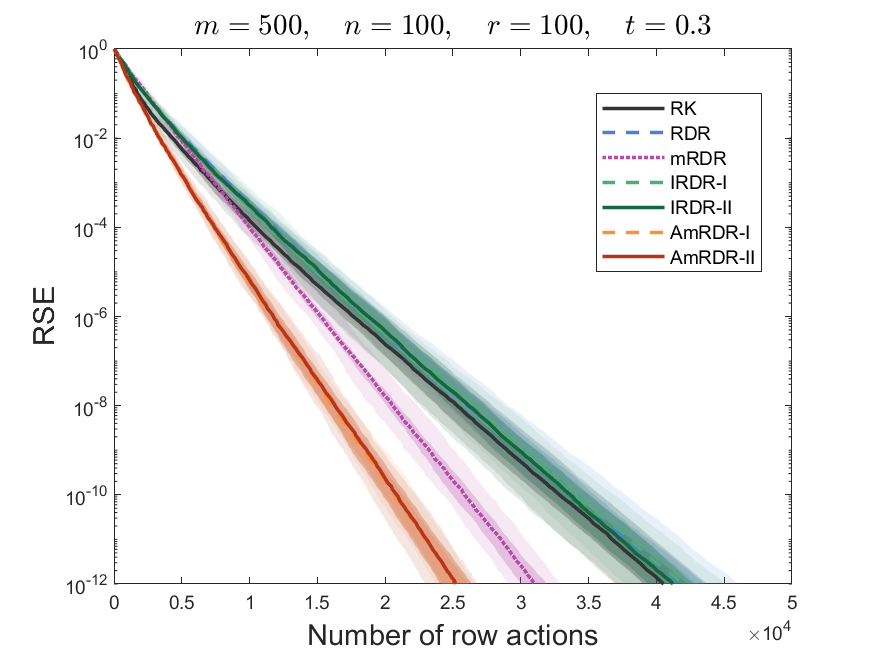}
		\includegraphics[width=0.33\linewidth]{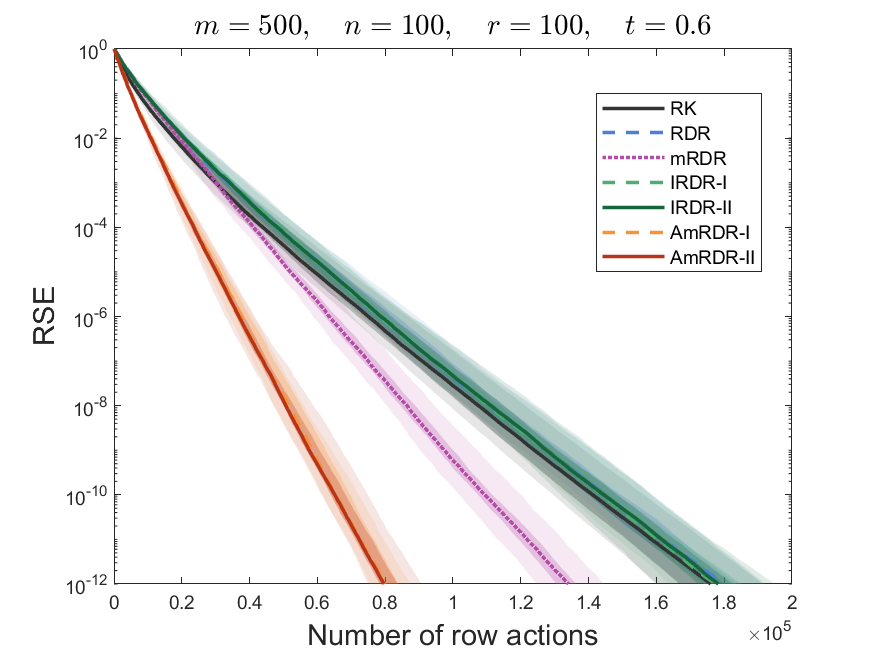}
		\includegraphics[width=0.33\linewidth]{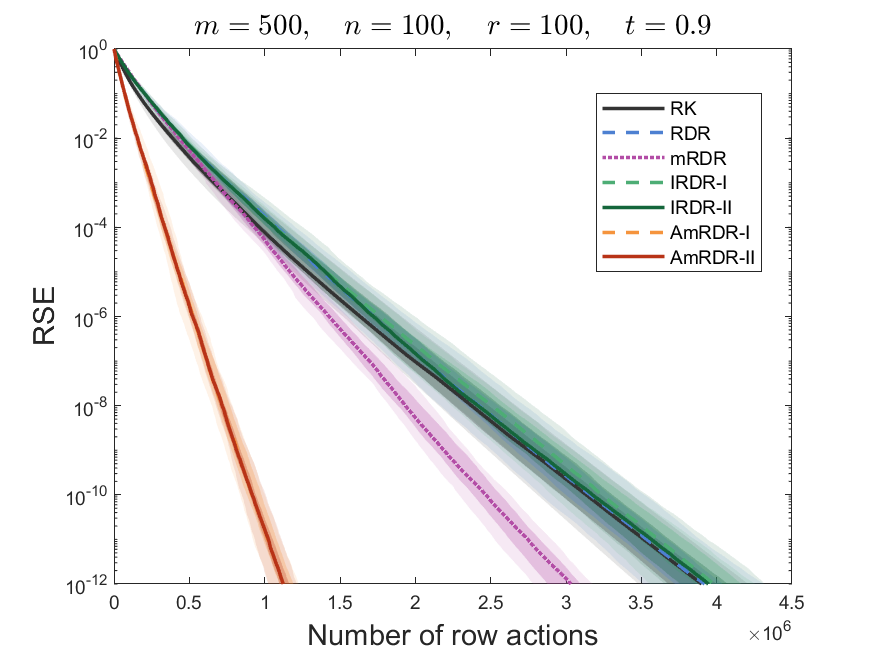}\\
		\includegraphics[width=0.33\linewidth]{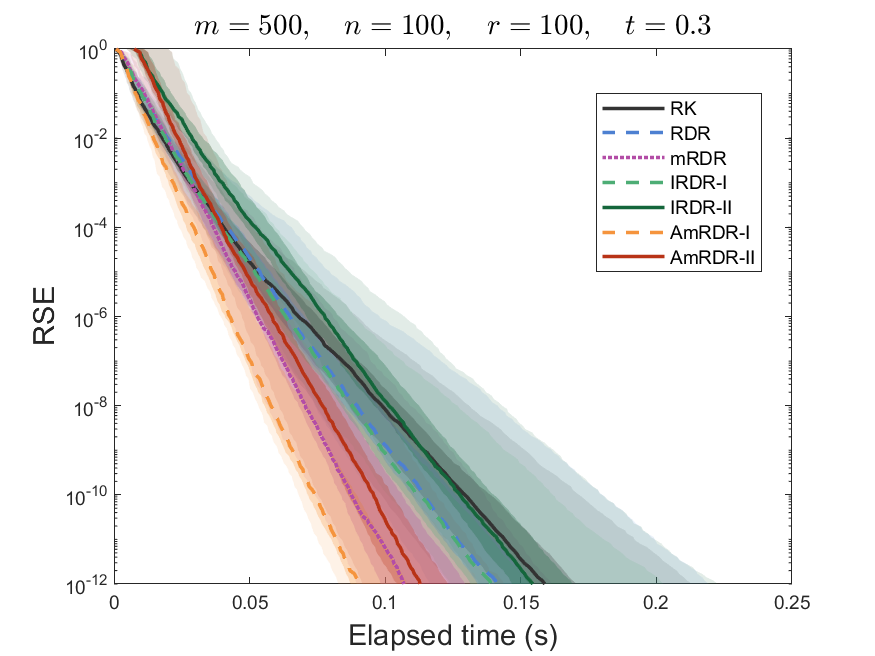}
		\includegraphics[width=0.33\linewidth]{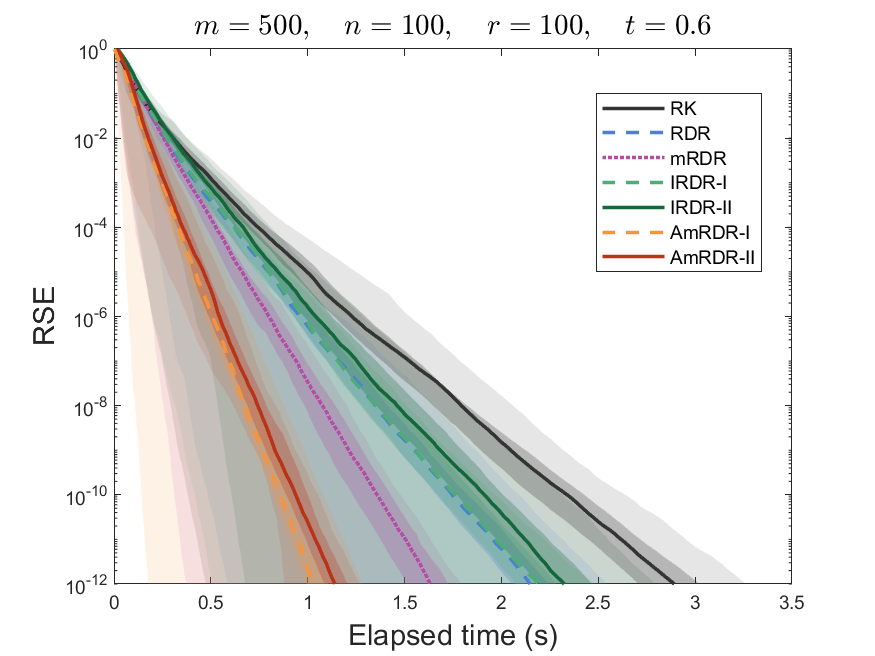}
		\includegraphics[width=0.33\linewidth]{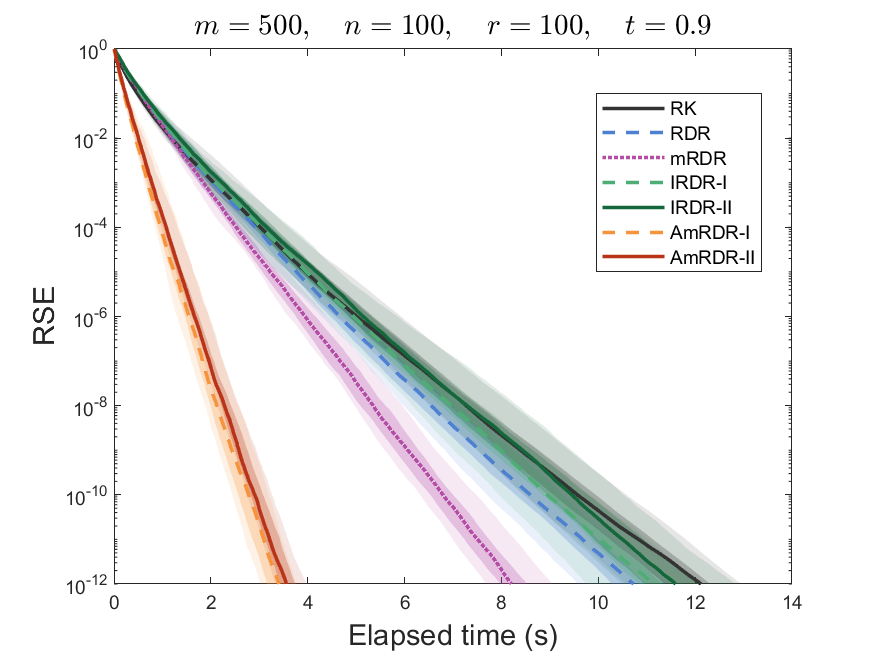}
	\end{tabular}
	\caption{Performance of RK, RDR, mRDR, { IRDR}-I, { IRDR}-II, { AmRDR}-I, and { AmRDR}-II for linear systems with uniform distribution matrix. Figures depict the number of row actions and the elapsed time (in seconds) vs RSE.  The title of each plot indicates the values of $m, n$, and $t$. }
	\label{figure3}
\end{figure}

\subsection{Comparison to the other methods: real-world test instances}

In this subsection, we evaluate real-world test instances from the SuiteSparse Matrix Collection \cite{kolodziej2019suitesparse}, including both full-rank and rank-deficient cases. We focus our comparison on mRDR, AmRDR-I, and AmRDR-II. Table~\ref{table1} presents the computational results, where an entry marked ``--'' indicates that at least one trial failed to reach the prescribed RSE tolerance.

For {\tt cari}, {\tt cage}, {\tt mk9-b1}, {\tt n4c5-b2}, and {\tt ch5-5-b2}, the successful methods require similar numbers of row actions, and RK achieves the shortest average elapsed time. For these instances, the additional computations associated with momentum and improved sampling do not yield a sufficient reduction in row actions to improve the overall runtime. For {\tt D\_8} and {\tt GL6\_D\_9}, mRDR achieves the lowest row-action counts and shortest elapsed times. On {\tt GL6\_D\_9}, AmRDR-II requires more row actions than mRDR but achieves a comparable elapsed time, even with preprocessing included.
The results also reveal limitations in the reliability of the adaptive methods: AmRDR-I fails in at least one trial on {\tt ch5-5-b2}, and AmRDR-II fails in at least one trial on {\tt D\_8}. Thus, the advantages observed on the synthetic test problems do not extend uniformly to these real-world instances. Although AmRDR-I and AmRDR-II eliminate manual tuning of the relaxation and momentum parameters, this benefit should be weighed against their computational overhead and observed failures. Since the tuning cost for mRDR is excluded, the reported timings compare performance with a preselected momentum parameter rather than the total cost of tuning and solving.

\begin{table}[p]
\centering
\rotatebox{90}{%
  \begin{minipage}{0.92\textheight}
  \centering
    \renewcommand\arraystretch{1.5}
    \setlength{\tabcolsep}{2pt}
    \caption{Average numbers of row actions (RA) and elapsed wall-clock times (Time, in seconds) for RK, mRDR, AmRDR-I, and AmRDR-II on linear systems with coefficient matrices from the SuiteSparse Matrix Collection \cite{kolodziej2019suitesparse}. All computations are terminated once $\mathrm{RSE}<10^{-12}$.}
    \label{table1}
    \centering
    {\scriptsize
        \begin{tabular}{  |c| c| c| c| c| c| c |c |c |c |c| c| c| c|}
            \hline
            \multirow{2}{*}{ Matrix}& \multirow{2}{*}{ $m\times n$ }  &\multirow{2}{*}{rank}& \multirow{2}{*}{$\frac{\sigma_{\max}(A)}{\sigma_{\min}(A)}$} &\multicolumn{2}{c| }{ RK} &\multicolumn{3}{c| }{mRDR}  &\multicolumn{2}{c| }{AmRDR-I} &\multicolumn{3}{c| }{ AmRDR-II}\\
            \cline{5-14}
            & &   &  & RA & Time   & RA & Time   & $\beta$  & RA & Time   & RA &  { PrepTime} & Time    \\
            \hline
            {\tt cari}  & $400 \times 1200$ & 400 & $3.13$  & 9.03e+3  & 0.1183 & 8.98e+3  &  0.1240 & 0.05 & 8.99e+3  &0.1464  & 9.03e+3 & 0.0151 & 0.1649\\
            \hline
            {\tt cage}  & $366 \times 366$ & 366  & $3.88$ & 8.20e+3& 0.0539 & $8.39\text{e}+3$  &   0.0566 & 0.05 & $8.23\text{e}+3$ & 0.0664  & $8.25\text{e}+3$ & 0.0022 & 0.0671\\
            \hline
            {\tt { mk9-b1}} & $378\times36$   &  35  & $2.51\text{e}+15$ & 9.06e+2& 0.0038 & $9.24\text{e}+2$ &   0.0043 & 0.05 & $9.19\text{e}+2$ &   0.0053 &  $9.10\text{e}+2$ &   0.0018  &  0.0069\\
            \hline
            {\tt { n4c5-b2}}  & $455 \times 105$ & 91 & $4.68\text{e}+15$ & 2.45e+2& 0.0131 & $2.46\text{e}+2$  & 0.0142  & 0.05 & $2.43\text{e}+2$ & 0.0146  & $2.44\text{e}+2$ & 0.0026 & 0.0197\\
            \hline
            {\tt { ch5-5-b2}}  & $600 \times 200$ & 176 & $5.45\text{e}+15$ & 5.19e+2& 0.0273 & 5.18e+2  & 0.0277 & 0.05 & -- & --  & $5.22\text{e}+2$ & 0.0039 & 0.0384\\
            \hline
            {\tt { D\_8}}  & $1132 \times 1271$ & 641 & $4.44\text{e}+31$ & 3.78e+4& 1.78 &  $3.05\text{e}+4$ &    1.47 & 0.3 &  $3.73\text{e}+4$ &   1.80  & -- & -- & --  \\
            \hline
            {\tt { GL6\_D\_9}}  & $340 \times 545$ & 220 & $1.10\text{e}+32$ & 1.51e+4& 0.6290&   $1.22\text{e}+4$ &    0.5671 & 0.35 &  $1.47\text{e}+4$ &   0.7400  & $1.42\text{e}+4$ &   0.0065 & 0.5928  \\
            \hline
        \end{tabular}
    }
  \end{minipage}%
}
\end{table}

\section{Concluding Remarks}\label{sec-concluding}
In this work, we improved the computational efficiency of the RDR method by introducing without-replacement sampling, volume sampling, and adaptive heavy-ball momentum. In particular, AmRDR adaptively determines the relaxation and momentum parameters from the iteration information without requiring their manual tuning. We establish linear convergence in expectation for the proposed methods. Numerical experiments verified the effectiveness of our method, particularly for linear systems in which the hyperplanes exhibit more pronounced angular differences.

There are numerous avenues for future research. DR splitting uses the resolvents of two operators, or projections onto two sets in feasibility problems, and is applicable when these operations can be evaluated efficiently. The present linear-system analysis provides a foundation for investigating dependent sampling and adaptive momentum in more general feasibility, nonsmooth composite, and constrained problems. Extending the proposed framework to these settings will require further convergence and stability analysis. Another promising direction is to develop extensions for sparse recovery, motivated by the randomized sparse Kaczmarz method in \cite{schopfer2019linear}. Furthermore, the application of randomized iterative methods for solving generalized absolute value equations (GAVE) has been explored in \cite{xie2024randomized}. Investigating the RDR method for solving GAVE is also a valuable topic.

\bibliographystyle{abbrv}
\bibliography{reference}

\section{Appendix. Proof of the main results}
	\label{appendix-P}
	\subsection{Omitted proofs in Section \ref{section3}}

We first state a lemma that will be useful in proving both Proposition \ref{propMN} and Theorem \ref{main-IRDR}.
	
	\begin{lemma}\label{lemma:symmetric_part}
		For any real square matrix $M\in\mathbb{R}^{m\times m}$, let $S_M = \frac{M+M^\top}{2}$ denote its symmetric part. Then for any nonzero vector $x\in\mathbb{R}^m$, we have
		\[ x^\top S_M x = x^\top Mx.\] 
	\end{lemma}
	\begin{proof}
		For any nonzero vector $x\in\mathbb{R}^m$, we have
		\[
		x^\top S_M x = \frac{1}{2}x^\top\left(M+M^\top\right)x = \frac{1}{2}x^\top Mx + \frac{1}{2}x^\top M^\top x.
		\]
		Note that the quadratic form $x^\top Mx$ is a scalar, hence its transpose coincides with itself, that is
		\[ x^\top M^\top x = \left(x^\top Mx\right)^\top = x^\top Mx.
		\]
		Substituting this equality into the above expression yields the desired result.
	\end{proof}

\begin{proof}[Proof of Proposition \ref{propMN}]
	We first prove that the matrix $S_M$ is positive definite. That is, for any nonzero vector \( x = (x_1, x_2, \dots, x_m)^\top \in \mathbb{R}^m \), we need to show that $x^\top S_M x>0$.
		By Lemma \ref{lemma:symmetric_part}, we have $x^\top S_M x = x^\top Mx$. Then we show that for any nonzero vector \( x = (x_1, x_2, \dots, x_m)^\top \in \mathbb{R}^m \), $x^\top Mx>0$.
	Let \( x = (x_1, x_2, \dots, x_m)^\top \in \mathbb{R}^m \), then we have
	$$
	\begin{aligned}
		x^{\top}M x
		=&(\Delta+1)x^{\top}x+
		x^\top
		\begin{bmatrix}
			-\frac{\Vert a_1 \Vert_2^2}{\Vert A \Vert_F^2 -\Vert a_1 \Vert_2^2}  & -\frac{2\left\langle a_1,a_2\right\rangle}{\Vert A \Vert_F^2 -\Vert a_2 \Vert_2^2}   & \cdots & -\frac{2\left\langle a_1,a_m\right\rangle}{\Vert A \Vert_F^2 -\Vert a_m \Vert_2^2} \\
			-\frac{2\left\langle a_1,a_2\right\rangle}{\Vert A \Vert_F^2 -\Vert a_1 \Vert_2^2}   & -\frac{\Vert a_2 \Vert_2^2}{\Vert A \Vert_F^2 -\Vert a_2 \Vert_2^2}   & \cdots & -\frac{2\left\langle a_2,a_m\right\rangle}{\Vert A \Vert_F^2 -\Vert a_m \Vert_2^2} \\
			\vdots & \vdots & \ddots & \vdots \\
			-\frac{2\left\langle a_1,a_m\right\rangle}{\Vert A \Vert_F^2 -\Vert a_1 \Vert_2^2}   & -\frac{2\left\langle a_2,a_m\right\rangle}{\Vert A \Vert_F^2 -\Vert a_2 \Vert_2^2}   & \cdots & -\frac{\Vert a_m \Vert_2^2}{\Vert A \Vert_F^2 -\Vert a_m \Vert_2^2}  
		\end{bmatrix}x\\
		=& 
		(\Delta+1)x^{\top}x-\sum_{i=1}^{m}\frac{\Vert a_i \Vert_2^2}{\Vert A \Vert_F^2 -\Vert a_i \Vert_2^2} x_i^2-\sum_{i=1}^{m}\sum\limits_{\substack{ j=1\\j\neq i}}^m\frac{2\left\langle a_i,a_j\right\rangle}{\Vert A \Vert_F^2 -\Vert a_j \Vert_2^2}x_ix_j
		\\
		\overset{(a)}{=}&\sum_{i=1}^{m}\left( \sum_{j=1}^m \frac{\|a_j\|_2^2}{\|A\|_F^2 - \|a_j\|_2^2} +1-\frac{\Vert a_i \Vert_2^2}{\Vert A \Vert_F^2 -\Vert a_i \Vert_2^2} \right)x_i^2-\sum_{i=1}^{m}\sum\limits_{\substack{ j=1\\j\neq i}}^m\frac{2\left\langle a_i,a_j\right\rangle}{\Vert A \Vert_F^2 -\Vert a_j \Vert_2^2}x_ix_j
		\\
		\overset{(b)}{=}&\sum_{i=1}^{m} \sum\limits_{\substack{ j=1\\j\neq i}}^m \left(\frac{\|a_j\|_2^2}{\|A\|_F^2 - \|a_j\|_2^2} +\frac{\Vert a_j \Vert_2^2}{\Vert A \Vert_F^2 -\Vert a_i \Vert_2^2} \right)x_i^2-\sum_{i=1}^{m}\sum\limits_{\substack{ j=1\\j\neq i}}^m\frac{2\left\langle a_i,a_j\right\rangle}{\Vert A \Vert_F^2 -\Vert a_j \Vert_2^2}x_ix_j
		\\
		=&\sum_{1\leq i<j\leq m}  \left(\frac{1}{\|A\|_F^2 - \|a_j\|_2^2} +\frac{1}{\Vert A \Vert_F^2 -\Vert a_i \Vert_2^2} \right)\left(\|a_j\|^2_2x_i^2+\|a_i\|^2_2x_j^2-2\langle a_i,a_j\rangle x_ix_j\right),
	\end{aligned}
	$$
	where $(a)$ follows from the definition of $\Delta= \sum_{j=1}^m \frac{\|a_j\|_2^2}{\|A\|_F^2 - \|a_j\|_2^2}$ and $(b)$ follows from  $1=\sum\limits_{\substack{ j=1\\j\neq i}}^m \frac{\Vert a_j \Vert_2^2 }{\Vert A \Vert_F^2 -\Vert a_i \Vert_2^2}$, $i=1,\dots,m$.
	
	We define 
	$$
	\mathcal{I}:=\{i\mid x_i\neq0\}, \ \mathcal{I}^c:=[m]\setminus \mathcal{I},
	$$
	and the pair sets
	$$
	\mathcal{J}_{+}:=\{(i,j)\mid x_ix_j>0, i<j, i,j\in\mathcal{I}\}, \ \mathcal{J}_{-}:=\{(i,j)\mid x_ix_j<0, i<j , i,j\in\mathcal{I}\}.
	$$
	Without loss of generality, assume \( \mathcal{I} = \{1, \dots, \ell\} \). Then,
	\[
	\begin{aligned}
		x^\top M x
		=&\ \sum_{\substack{i,j \in \mathcal{I}\\ i<j}} \left( \frac{1}{\|A\|_F^2 - \|a_j\|_2^2} + \frac{1}{\|A\|_F^2 - \|a_i\|_2^2} \right) \left( \|a_j\|_2^2 x_i^2 + \|a_i\|_2^2 x_j^2 - 2\langle a_i,a_j\rangle x_i x_j \right) \\
		&+ \sum_{\substack{i \in \mathcal{I},\, j \in \mathcal{I}^c}} \left( \frac{1}{\|A\|_F^2 - \|a_j\|_2^2} + \frac{1}{\|A\|_F^2 - \|a_i\|_2^2} \right) \left( \|a_j\|_2^2 x_i^2 + \|a_i\|_2^2 x_j^2 - 2\langle a_i,a_j\rangle x_i x_j \right) \\
		&+ \sum_{\substack{i,j \in \mathcal{I}^c\\ i<j}} \left( \frac{1}{\|A\|_F^2 - \|a_j\|_2^2} + \frac{1}{\|A\|_F^2 - \|a_i\|_2^2} \right) \left( \|a_j\|_2^2 x_i^2 + \|a_i\|_2^2 x_j^2 - 2\langle a_i,a_j\rangle x_i x_j \right).
	\end{aligned}
	\]
	Note that all terms involving \( x_j = 0 \) for \( j \in \mathcal{I}^c \) vanish. Thus, we obtain
	\[
	\begin{aligned}
		x^\top M x =&\sum_{\substack{i,j \in \mathcal{I}\\ i<j}} \left( \frac{1}{\|A\|_F^2 - \|a_j\|_2^2} + \frac{1}{\|A\|_F^2 - \|a_i\|_2^2} \right) \left( \|a_j\|_2^2 x_i^2 + \|a_i\|_2^2 x_j^2 - 2\langle a_i,a_j\rangle x_i x_j \right) \\
		&+ \sum_{\substack{i \in \mathcal{I},\, j \in \mathcal{I}^c}} \left( \frac{1}{\|A\|_F^2 - \|a_j\|_2^2} + \frac{1}{\|A\|_F^2 - \|a_i\|_2^2} \right)\|a_j\|_2^2 x_i^2
		\\
		\geq&\ \sum_{(i,j)\in \mathcal{J}_+} \left( \frac{1}{\|A\|_F^2 - \|a_j\|_2^2} + \frac{1}{\|A\|_F^2 - \|a_i\|_2^2} \right) \left( \|a_j\|_2 x_i - \|a_i\|_2 x_j \right)^2 \\
		&+ \sum_{(i,j)\in \mathcal{J}_-} \left( \frac{1}{\|A\|_F^2 - \|a_j\|_2^2} + \frac{1}{\|A\|_F^2 - \|a_i\|_2^2} \right) \left( \|a_j\|_2 x_i + \|a_i\|_2 x_j \right)^2 \\
		&+ \sum_{\substack{i \in \mathcal{I},\, j \in \mathcal{I}^c}} \left( \frac{1}{\|A\|_F^2 - \|a_j\|_2^2} + \frac{1}{\|A\|_F^2 - \|a_i\|_2^2} \right) \|a_j\|_2^2 x_i^2 \\
		\geq&\ 0.
	\end{aligned}
	\]
	The first inequality follows from the Cauchy–Schwarz inequality, and equality holds if and only if \( \langle a_i, a_j \rangle^2 = \|a_i\|_2^2 \|a_j\|_2^2 \) for all \( i \ne j \in \mathcal{I} \), i.e., the vectors \( a_i \) (for \( i \in \mathcal{I} \)) are colinear. The last inequality becomes equality only if \( a_j = 0 \) for all \( j \in \mathcal{I}^c \).
	Therefore, \( x^\top M x = 0 \) if and only if all \( a_i \) (for \( i \in \mathcal{I} \)) are colinear and \( a_j = 0 \) for all \( j \in \mathcal{I}^c \), which implies \( \operatorname{rank}(A) \leq 1 \). Since we assume \( \operatorname{rank}(A) \geq 2 \), it follows that \( x^\top M x > 0 \) for all \( x \ne 0 \). Correspondingly, \( x^\top S_M x > 0 \) for all \( x \ne 0 \), and hence symmetric matrix $S_M$ is positive definite.

	Next, we prove $N$	is positive definite.
	For any nonzero vector \( x = (x_1, x_2, \dots, x_m)^\top \in \mathbb{R}^m \), we have
	$$
	\begin{aligned}
		x^{\top}Nx=&x^\top \begin{bmatrix}
			\sum\limits_{j=1}^m g_{1,j} \|a_j\|_2^2 & -g_{1,2} \langle a_1, a_2 \rangle & \cdots & -g_{1,m} \langle a_1, a_m \rangle \\
			-g_{2,1} \langle a_2, a_1 \rangle & \sum\limits_{j=1}^m g_{2,j} \|a_j\|_2^2 & \cdots & -g_{2,m} \langle a_2, a_m \rangle \\
			\vdots & \vdots & \ddots & \vdots \\
			-g_{m,1} \langle a_m, a_1 \rangle & -g_{m,2} \langle a_m, a_2 \rangle & \cdots & \sum\limits_{j=1}^m g_{m,j} \|a_j\|_2^2
		\end{bmatrix} x\\
		=& \sum_{i=1}^{m} \left(\sum\limits_{\substack{ j=1\\j\neq i}}^m \|a_j\|^2_2g_{j,i}\right)x^2_i- \sum_{i=1}^{m} \sum\limits_{\substack{ j=1\\j\neq i}}^m g_{j,i}\langle a_i,a_j\rangle x_ix_j\\
		=& \sum_{1\leq i<j\leq m}g_{j,i} \left(\|a_j\|^2_2x_i^2+\|a_i\|^2_2x_j^2-2\langle a_i,a_j\rangle x_ix_j\right).
	\end{aligned}
	$$
	Similar to the discussion above, we can conclude that { symmetric matrix} $N$ is positive definite.
\end{proof}

Next, we will prove Theorem \ref{main-IRDR}. First, we introduce some useful lemmas. For any $i\in[m]$, we define
\begin{equation}\label{reflection operator}
	T_{\mathcal{H}_i}:=I-2\frac{a_{i}a_i^{\top}}{\|a_i\|^2_2}.
\end{equation}
It is easy to verify that $T_{\mathcal{H}_i}T_{\mathcal{H}_i}=I$, and thus for any $y\in \mathbb{R}^n$, $\|T_{\mathcal{H}_i}y\|_2=\|y\|_2$.

\begin{lemma}\label{exp-reflection-operator}
	Let \( T_{\mathcal{H}_i} \) be defined in \eqref{reflection operator}. 
	If indices \( i, j \in [m] \) are selected according to Strategy I, then 
	\[
	\mathbb{E}\left[ T_{\mathcal{H}_j} T_{\mathcal{H}_i} \right] = I - \frac{2}{\|A\|_F^2} A^\top M A,
	\]
	where the matrix \( M \) is defined in \eqref{definition-M}.
	If indices \( i, j \in [m] \) are selected according to Strategy II, then we have
	\[
	\mathbb{E}\left[ T_{\mathcal{H}_j} T_{\mathcal{H}_i} \right] = I - \frac{4}{\|A\|_F^4 - \|A A^\top\|_F^2} A^\top N A,
	\]
	where the matrix \( N \) is defined in \eqref{definition-N}.
\end{lemma}

\begin{proof}
	We begin by computing the expected composition \( \mathbb{E}\left[ T_{\mathcal{H}_j} T_{\mathcal{H}_i} \right] \) under Strategy I. We have
	\begin{equation}\label{exp-2.4}
		\begin{aligned}
			&\mathop{\mathbb{E}}\left[T_{\mathcal{H}_j}T_{\mathcal{H}_i}\right]
			=\sum\limits_{i=1}^m \frac{\Vert a_i \Vert_2^2}{\Vert A \Vert_F^2} \ \sum\limits_{\substack{ j=1\\j\neq i}}^m \frac{\Vert a_j \Vert_2^2}{\Vert A \Vert_F^2 -\Vert a_i \Vert_2^2} 
			\left(I-2\frac{a_ja_j^{\top}}{\|a_j\|^2_2}\right)
			\left(I-2\frac{a_ia_i^{\top}}{\|a_i\|^2_2}\right)\\
			=&I-\underbrace{2\sum\limits_{i=1}^m \frac{\Vert a_i \Vert_2^2}{\Vert A \Vert_F^2}\  \sum\limits_{\substack{ j=1\\j\neq i}}^m \frac{\Vert a_j \Vert_2^2}{\Vert A \Vert_F^2 -\Vert a_i \Vert_2^2}\ \frac{a_ja_j^{\top}}{\|a_j\|^2_2}}_{\textcircled{a}}
			-\underbrace{2\sum\limits_{i=1}^m \frac{\Vert a_i \Vert_2^2}{\Vert A \Vert_F^2}\  \sum\limits_{\substack{ j=1\\j\neq i}}^m \frac{\Vert a_j \Vert_2^2}{\Vert A \Vert_F^2 -\Vert a_i \Vert_2^2}\ \frac{a_ia_i^{\top}}{\|a_i\|^2_2}}_{\textcircled{b}}\\
			&+\underbrace{4\sum\limits_{i=1}^m \frac{\Vert a_i \Vert_2^2}{\Vert A \Vert_F^2}\  \sum\limits_{\substack{ j=1\\j\neq i}}^m \frac{\Vert a_j \Vert_2^2}{\Vert A \Vert_F^2 -\Vert a_i \Vert_2^2}\ \frac{a_ja_j^{\top}}{\|a_j\|^2_2}\ \frac{a_ia_i^{\top}}{\|a_i\|^2_2}}_{\textcircled{c}}
		\end{aligned}
	\end{equation}
	We now analyze the three expressions $\textcircled{a},\textcircled{b}$, and $\textcircled{c}$ separately. For term \( \textcircled{a} \), we have
	$$
	\begin{aligned}
		\textcircled{a}=&2\sum\limits_{i=1}^m \frac{\Vert a_i \Vert_2^2}{\Vert A \Vert_F^2}\  \sum\limits_{\substack{ j=1\\j\neq i}}^m \frac{\Vert a_j \Vert_2^2}{\Vert A \Vert_F^2 -\Vert a_i \Vert_2^2}\ \frac{a_ja_j^{\top}}{\|a_j\|^2_2}
		=\frac{2}{\Vert A \Vert_F^2}
		\sum\limits_{i=1}^m \left(\frac{\Vert a_i \Vert_2^2 }{\Vert A \Vert_F^2 -\Vert a_i \Vert_2^2}\ \sum\limits_{\substack{ j=1\\j\neq i}}^m a_ja_j^{\top}\right)\\
		=&\frac{2}{\Vert A \Vert_F^2}
		\sum\limits_{i=1}^m \left(\sum\limits_{\substack{ j=1\\j\neq i}}^m \frac{\Vert a_j \Vert_2^2 }{\Vert A \Vert_F^2 -\Vert a_j \Vert_2^2}\right)a_ia_i^{\top}=2\frac{A^{\top}\hat{D}A}{\Vert A \Vert_F^2},
	\end{aligned}
	$$
	where the third equality follows by regrouping terms based on the coefficients of \( a_i a_i^\top \), and the last equality uses the definition
	$
	\hat{D} \coloneqq \operatorname{diag}(\hat{d}_1, \ldots, \hat{d}_m)$ with $ \hat{d}_i = \sum_{\substack{j=1 \\ j \neq i}}^m \frac{\|a_j\|_2^2}{\|A\|_F^2 - \|a_j\|_2^2}.
	$
	For term $\textcircled{b}$, we have
	\begin{align*}
		\textcircled{b} =&2\sum\limits_{i=1}^m \frac{\Vert a_i \Vert_2^2}{\Vert A \Vert_F^2}\  \sum\limits_{\substack{ j=1\\j\neq i}}^m \frac{\Vert a_j \Vert_2^2}{\Vert A \Vert_F^2 -\Vert a_i \Vert_2^2}\ \frac{a_ia_i^{\top}}{\|a_i\|^2_2}\\
		=& 2\sum\limits_{i=1}^m \frac{a_ia_i^{\top}}{\Vert A \Vert_F^2}\  \sum\limits_{\substack{ j=1\\j\neq i}}^m \frac{\Vert a_j \Vert_2^2}{\Vert A \Vert_F^2 -\Vert a_i \Vert_2^2}
		=2\sum\limits_{i=1}^m \frac{a_ia_i^{\top}}{\Vert A \Vert_F^2}\
		=\ 2\frac{A^{\top}A}{\Vert A \Vert_F^2}.
	\end{align*}   
	For term $\textcircled{c}$, we have
	\begin{equation}
		\begin{aligned}
			\textcircled{c}=&4\sum\limits_{i=1}^m \frac{\Vert a_i \Vert_2^2}{\Vert A \Vert_F^2}\  \sum\limits_{\substack{ j=1\\j\neq i}}^m \frac{\Vert a_j \Vert_2^2}{\Vert A \Vert_F^2 -\Vert a_i \Vert_2^2}\ \frac{a_ja_j^{\top}}{\|a_j\|^2_2}\ \frac{a_ia_i^{\top}}{\|a_i\|^2_2}
			=4\sum\limits_{i=1}^m   \sum\limits_{\substack{ j=1\\j\neq i}}^m\ 
			\frac{a_ja_j^{\top}}{\Vert A \Vert_F^2 -\Vert a_i \Vert_2^2}\ \frac{a_ia_i^{\top}}{\Vert A \Vert_F^2}\\
			=&4\sum\limits_{i=1}^m  \frac{A^{\top}A-a_ia_i^{\top}}{\Vert A \Vert_F^2 -\Vert a_i \Vert_2^2}\frac{a_ia_i^{\top}}{\Vert A \Vert_F^2}
			=\frac{4}{\Vert A \Vert_F^2} \left(A^{\top}A \sum\limits_{i=1}^m  \frac{a_ia_i^{\top}}{\Vert A \Vert_F^2 -\Vert a_i \Vert_2^2} \     -      \sum\limits_{i=1}^m  \frac{a_ia_i^{\top}a_ia_i^{\top}}{\Vert A \Vert_F^2 -\Vert a_i \Vert_2^2}\right)\\
			=&\frac{4}{\Vert A \Vert_F^2} \left(A^{\top}A\sum\limits_{i=1}^m  \frac{1}{\Vert A \Vert_F^2 -\Vert a_i \Vert_2^2}a_ia_i^{\top} \     -      \sum\limits_{i=1}^m  \frac{\Vert a_i \Vert_2^2}{\Vert A \Vert_F^2 -\Vert a_i \Vert_2^2}\ a_ia_i^{\top}\right) \\
			=& 4\frac{A^{\top}AA^{\top} \bar{D}  A  - A^{\top}DA }{\Vert A \Vert_F^2},
		\end{aligned}
		\nonumber
	\end{equation}
	where $D\coloneqq \operatorname{diag}(d_1,\ldots,d_m)$ with $d_i=\frac{\Vert a_i \Vert_2^2}{\Vert A \Vert_F^2 -\Vert a_i \Vert_2^2}$, and  $\bar{D} \coloneqq \operatorname{diag}(\bar{d_1},\ldots,\bar{d_m})$ with $\bar{d_i}=\frac{1}{\Vert A \Vert_F^2 -\Vert a_i \Vert_2^2}$.
	By substituting all these equations into \eqref{exp-2.4}, we obtain
	$$\label{exp-result}
	\begin{aligned}
		\mathop{\mathbb{E}}\left[T_{\mathcal{H}_j}T_{\mathcal{H}_i}\right]
		=&I-2\frac{A^{\top} \left(\hat{D}+2D\right)A}{\Vert A \Vert_F^2}
		-2\frac{A^{\top}A}{\Vert A \Vert_F^2}
		+4\frac{A^{\top}AA^{\top}\bar{D}A }{\Vert A \Vert_F^2} \\
		=&I-2\frac{A^{\top} \left(\Delta I+D\right)A}{\Vert A \Vert_F^2}
		-2\frac{A^{\top}A}{\Vert A \Vert_F^2}
		+4\frac{A^{\top}AA^{\top}\bar{D}A }{\Vert A \Vert_F^2}
		=I-\frac{2}{\|A\|^2_F}A^{\top}
		M A,
	\end{aligned}
	$$
	where the second equality comes from the reality that $\hat{d_i}+d_i =\Delta   =  \sum\limits_{j=1}^m \frac{\Vert a_j \Vert_2^2}{\Vert A \Vert_F^2 -\Vert a_j \Vert_2^2}$.
	
	Next we compute the expected composition \( \mathbb{E}\left[ T_{\mathcal{H}_j} T_{\mathcal{H}_i} \right] \) under Strategy II. For an ordered pair $(i,j)\in[m]^2$, Definition~\ref{vs} gives
\[
\mathbb P\big(\mathcal S_k=(i,j)\big)=
\frac{\Vert a_j\Vert_2^2\Vert a_i\Vert_2^2-\langle a_i,a_j\rangle^2}
{\Vert A\Vert_F^4-\Vert AA^\top\Vert_F^2},
\]
as
$
\sum_{i=1}^m\sum_{j=1}^m\left(\Vert a_j\Vert_2^2\Vert a_i\Vert_2^2-\langle a_i,a_j\rangle^2\right)
=\Vert A\Vert_F^4-\Vert AA^\top\Vert_F^2.
$
Hence,\begin{equation}
\label{exp-zhankaishi}
\begin{aligned}
\mathop{\mathbb{E}}\left[ T_{\mathcal{H}_{j}}T_{\mathcal{H}_{i}}\right]
=&\sum_{i=1}^m\sum_{j=1}^m
\frac{\Vert a_j\Vert_2^2\Vert a_i\Vert_2^2-\left\langle a_i,a_j\right\rangle^2}
{\Vert A\Vert_F^4-\Vert AA^\top\Vert_F^2}
\left(I-2\frac{a_ja_j^\top}{\|a_j\|_2^2}\right)
\left(I-2\frac{a_ia_i^\top}{\|a_i\|_2^2}\right)\\
=& I
\underbrace{-2 \sum\limits_{i=1}^m \sum\limits_{j=1}^m 
\frac{\Vert a_j\Vert_2^2 \Vert a_i \Vert_2^2-\left\langle a_i,a_j \right\rangle^2}
{\Vert A \Vert_F^4 -\Vert AA^{\top} \Vert_F^2}
\frac{a_ja_j^{\top}}{\|a_j\|^2_2} }_{\textcircled{e}}\\
&\underbrace{-2\sum\limits_{i=1}^m \sum\limits_{j=1}^m 
\frac{\Vert a_j\Vert_2^2 \Vert a_i \Vert_2^2-\left\langle a_i,a_j \right\rangle^2}
{\Vert A \Vert_F^4 -\Vert AA^{\top} \Vert_F^2}
\frac{a_ia_i^{\top}}{\|a_i\|^2_2}}_{\textcircled{f}}\\
&\underbrace{+4\sum\limits_{i=1}^m \sum\limits_{j=1}^m
\frac{\Vert a_j\Vert_2^2 \Vert a_i \Vert_2^2-\left\langle a_i,a_j \right\rangle^2}
{\Vert A \Vert_F^4 -\Vert AA^{\top} \Vert_F^2}
\frac{a_ja_j^{\top}}{\|a_j\|^2_2}
\frac{a_ia_i^{\top}} {\|a_i\|^2_2}}_{\textcircled{g}}.
\end{aligned}
\end{equation}
	We first analyze the expression $\textcircled{e}$. We have 
	\begin{equation}
		\begin{aligned}\label{a}
			\textcircled{e}=&-2\sum\limits_{i=1}^m \ 
			\sum\limits_{j=1}^m 
			\frac{\Vert a_j\Vert_2^2 \Vert a_i \Vert_2^2   -   \left\langle a_i,a_j \right\rangle^2} 
			{ \Vert A \Vert_F^4 -\Vert AA^{\top} \Vert_F^2} 
			\frac{a_ja_j^{\top}}{\|a_j\|^2_2} 
			=
			-2\sum\limits_{i=1}^m \ 
			\sum\limits_{j=1}^m
			\frac{g_{i,j}\Vert a_i \Vert_2^2 }
			{ \Vert A \Vert_F^4 -\Vert AA^{\top} \Vert_F^2}
			a_ja_j^{\top} 
			\\
			=&-\frac{2}
			{\Vert A \Vert_F^4 -\Vert AA^{\top} \Vert_F^2}
			\sum\limits_{i=1}^m
			\left(A^{\top} \left(\Vert a_i \Vert_2^2 G^i\right)A \right) 
			=-\frac{2}
			{\Vert A \Vert_F^4 -\Vert AA^{\top} \Vert_F^2}
			A^{\top} GA,
		\end{aligned}
	\end{equation}
	where matrix $G^i\coloneqq \operatorname{diag}(g_{i,1},\ldots,g_{i,m})$ with $g_{i,j}=1-\frac{\langle a_i,a_j \rangle^2}{\|a_i\|^2_2\|a_j\|^2_2}$, and
	$G\coloneqq \sum\limits_{i=1}^m \Vert a_i \Vert_2^2\ G^i = \operatorname{diag}(g_1,\ldots,g_m)$ with $ g_i=\sum\limits_{j=1}^m  g_{i,j}\|a_j\|^2_2  ,\ i,j=1,\ldots, m.$
	It is easily to see that  $\textcircled{e}=\textcircled{f}$.
	Next, we analyze the expression $\textcircled{g}$. We have
	\begin{equation}
		\begin{aligned}\label{c}
			\textcircled{g}=&4\sum\limits_{i=1}^m \ 
			\sum\limits_{j=1}^m
			\frac{\Vert a_j\Vert_2^2 \Vert a_i \Vert_2^2   -   \left\langle a_i,a_j \right\rangle^2}
			{ \Vert A \Vert_F^4 -\Vert AA^{\top} \Vert_F^2}
			\frac{a_ja_j^{\top}}{\|a_j\|^2_2}
			\frac{a_ia_i^{\top}} {\|a_i\|^2_2}
			=
			\frac{4}
			{ \Vert A \Vert_F^4 -\Vert AA^{\top} \Vert_F^2}
			\sum\limits_{i=1}^m \ 
			\sum\limits_{j=1}^m
			g_{i,j}
			a_ja_j^{\top}
			a_ia_i^{\top} \\
			=&
			\frac{4}
			{ \Vert A \Vert_F^4 -\Vert AA^{\top} \Vert_F^2}
			\sum\limits_{i=1}^m  
			A^{\top} G^iA 
			a_ia_i^{\top}
			=
			\frac{4}
			{ \Vert A \Vert_F^4 -\Vert AA^{\top} \Vert_F^2}
			A^{\top}\bar{G}A,
		\end{aligned}
	\end{equation}
	where $\bar{G} \coloneqq (\bar{g_1},\ldots,\bar{g_m})$ with $\bar{g_i}=G^iAa_i$.
	By substituting \eqref{a} and \eqref{c} into \eqref{exp-zhankaishi}, we obtain
	\[
			\mathop{\mathbb{E}}\left[ T_{\mathcal{H}_{j}}T_{\mathcal{H}_{i}}\right]
			=
			I
			-\frac{4}
			{ \Vert A \Vert_F^4 -\Vert AA^{\top} \Vert_F^2}
			A^{\top}
			\left(G-\bar{G}\right)A 
			=
			I-\frac{4}{\|A\|^4_F - \|AA^{\top}\|^2_F}A^{\top}
			N A,
	\]
	where matrix $N=G-\bar{G}$ is given by \eqref{definition-N}.
\end{proof}

\begin{lemma}[\cite{han2024randomized}, Lemma A.2]
	\label{lemma-22}
	Let $\{x^k\}_{k\geq0}$ and $\{z^k\}_{\geq 0}$ be the sequences generated by Algorithm \ref{alg:IRDR}. Then
	$$
	\|(1-\alpha)x^{k}+\alpha z^k-x^*\|^2_2=\big(\alpha^2+(1-\alpha)^2\big)\|x^k-x^*\|^2_2
	+2\alpha(1-\alpha)\langle z^k-x^*,x^k-x^*\rangle.
	$$
\end{lemma}

\begin{lemma}[\cite{han2024randomized}, Lemma A.3]
	\label{lemma-23}
	Let $\{x^k\}_{k\geq 0}$ be the sequence generated by Algorithm \ref{alg:IRDR} and $x_*^0=A^{\dagger}b+(I-A^\dagger A)x^0$. Then $ x^k - x_*^0 \in \operatorname{Range}(A^\top) $.
\end{lemma}

Now, we are ready to prove Theorem \ref{main-IRDR}.

\begin{proof}[Proof of Theorem \ref{main-IRDR}]
	By taking the conditional expectation, we obtain
	\begin{equation}
		\begin{aligned}\label{exp-inner product}
			&\mathop{\mathbb{E}}\big[\|x^{k+1}-x_*^0\|^2_2 \mid x^k\big]=
			\mathop{\mathbb{E}}\big[\|(1-\alpha)x^{k}+\alpha z^k-x_*^0\|^2_2|x^k\big]\\
			=&\left(\alpha^2+(1-\alpha)^2\right)\|x^k-x_*^0\|^2_2
			+2\alpha(1-\alpha)
			\mathop{\mathbb{E}}\left[\left\langle T_{\mathcal{H}_{i_{k_2}}}T_{\mathcal{H}_{i_{k_1}}}(x^k-x_*^0),x^k-x_*^0\right\rangle\right]\\
			=&\left(\alpha^2+(1-\alpha)^2\right)\|x^k-x_*^0\|^2_2
			+2\alpha(1-\alpha)
			\left\langle \mathop{\mathbb{E}}\left[ T_{\mathcal{H}_{i_{k_2}}}T_{\mathcal{H}_{i_{k_1}}}\right](x^k-x_*^0),x^k-x_*^0\right\rangle ,
		\end{aligned}
	\end{equation}
	where the first equality follows from Step $3$ in Algorithm \ref{alg:IRDR}, the second equality follows from Lemma \ref{lemma-22}, and the last equality follows from the  linearity of the expectation.
	
	First, we consider Algorithm \ref{alg:IRDR} with  Strategy I.
	According to { Lemma \ref{lemma:symmetric_part}}, Lemma \ref{exp-reflection-operator} and  Proposition~\ref{propMN}, we have
	\begin{equation}
		\begin{aligned}
			\label{thm-qyz-1}
			\langle \mathop{\mathbb{E}}[ T_{\mathcal{H}_{i_{k_2}}}T_{\mathcal{H}_{i_{k_1}}}](x^k-x_*^0),x^k-x_*^0\rangle 
			= \|x^k-x_*^0\|^2_2-\frac{2}{\|A\|^2_F} 
			\|	{ S_M}^{\frac{1}{2}}
			A(x^k-x_*^0 )\|^2_2.
		\end{aligned}
	\end{equation}
	Moreover, by Lemma~\ref{lemma-23}, we know that \( x^k - x_*^0 \in \operatorname{Range}(A^\top) \). Since { $S_M$} is positive definite, it follows that \( x^k - x_*^0 \in \operatorname{Range}(A^\top { S_M}^\frac{1}{2}) \). Therefore, we have the inequality
	$$		\|	{ S_M}^{\frac{1}{2}}
	A(x^k-x_*^0 )\|^2_2
	\geq 
	\sigma_{\min}^2({ S_M}^{\frac{1}{2}}A)\|x^k-x_*^0\|^2_2.$$
	Substituting this bound and \eqref{thm-qyz-1} into \eqref{exp-inner product}, we obtain
	\begin{equation}
		\begin{aligned}
			\mathop{\mathbb{E}}[\|x^{k+1}-x_*^0\|^2_2\mid x^k] 
			\leq \left(1-4\alpha(1-\alpha)\frac{\sigma_{\min}^2({ S_M}^{\frac{1}{2}}A)}{\|A\|^2_F}\right)
			\|x^k-x_*^0\|^2_2.
		\end{aligned}
		\nonumber
	\end{equation}
	Taking the expectation over the entire history yields the desired linear convergence result. 
	
	The proof under Strategy II follows analogously by replacing the matrix \( { S_M} \) with the corresponding matrix \( N \) for Strategy II. The same argument yields the desired result.
\end{proof}

{ 
	\begin{proof}[Proof of Proposition \ref{IRDR-SigleStep}]
		Let \( x = (x_1, x_2, \dots, x_m)^\top \in \mathbb{R}^m \). Then we have
		$$
		\begin{aligned}
			& x^\top \left(M-2\left(I- \frac{AA^\top}{\|A\|_F^2}\right)\right) x \\
			=&
			x^\top
			\begin{bmatrix}
				\omega_1  & -\frac{2\Vert a_2 \Vert_2^2\left\langle a_1,a_2\right\rangle}{\Vert A \Vert_F^2\left(\Vert A \Vert_F^2 -\Vert a_2 \Vert_2^2\right)}   & \cdots & -\frac{2\Vert a_m \Vert_2^2\left\langle a_1,a_m\right\rangle}{\Vert A \Vert_F^2\left(\Vert A \Vert_F^2 -\Vert a_m \Vert_2^2\right)} \\
				-\frac{2\Vert a_1 \Vert_2^2\left\langle a_1,a_2\right\rangle}{\Vert A \Vert_F^2\left(\Vert A \Vert_F^2 -\Vert a_1 \Vert_2^2\right)}   & \omega_2   & \cdots & -\frac{2\Vert a_m \Vert_2^2\left\langle a_2,a_m\right\rangle}{\Vert A \Vert_F^2\left(\Vert A \Vert_F^2 -\Vert a_m \Vert_2^2\right)} \\
				\vdots & \vdots & \ddots & \vdots \\
				-\frac{2\Vert a_1 \Vert_2^2\left\langle a_1,a_m\right\rangle}{\Vert A \Vert_F^2\left(\Vert A \Vert_F^2 -\Vert a_1 \Vert_2^2\right)}   & -\frac{2\Vert a_2 \Vert_2^2\left\langle a_2,a_m\right\rangle}{\Vert A \Vert_F^2\left(\Vert A \Vert_F^2 -\Vert a_2 \Vert_2^2\right)}   & \cdots & \omega_m  
			\end{bmatrix}x\\
			=& 
			\sum_{i=1}^{m}\omega_i x_i^2
			-\sum_{i=1}^{m}\sum\limits_{\substack{ j=1\\j > i}}^m
			\frac{2\left\langle a_i,a_j\right\rangle}{\Vert A \Vert_F^2}
			\left(\frac{\Vert a_i \Vert_2^2}{\Vert A \Vert_F^2-\Vert a_i \Vert_2^2} + \frac{\Vert a_j \Vert_2^2}{\Vert A \Vert_F^2-\Vert a_j \Vert_2^2} \right)
			x_ix_j, \\
		\end{aligned}
		$$
		where 
		\begin{align*}
			\omega_i:=&\frac{\Vert a_i \Vert_2^2}{\|A\|_F^2} + \sum\limits_{\substack{ j=1\\j\neq i}}^m\left( \frac{\Vert a_j \Vert_2^2}{\Vert A \Vert_F^2 -\Vert a_j \Vert_2^2} - \frac{\Vert a_j \Vert_2^2}{\|A\|_F^2}\right)\\
			= &\sum\limits_{\substack{ j=1\\j \neq i}}^m\frac{\Vert a_j \Vert_2^2}{\Vert A \Vert_F^2} \left(\frac{\Vert a_i \Vert_2^2}{\Vert A \Vert_F^2-\Vert a_i \Vert_2^2} + \frac{\Vert a_j \Vert_2^2}{\Vert A \Vert_F^2-\Vert a_j \Vert_2^2} \right).
		\end{align*}
		Substituting the above expression back into the quadratic form, we can get
		$$
		\begin{aligned}
			&x^\top \left(M-2\left(I- \frac{AA^\top}{\|A\|_F^2}\right)\right) x \\
			=&\sum_{i=1}^{m} \left(\sum\limits_{\substack{ j=1\\j \neq i}}^m\frac{\Vert a_j \Vert_2^2}{\Vert A \Vert_F^2} \left(\frac{\Vert a_i \Vert_2^2}{\Vert A \Vert_F^2-\Vert a_i \Vert_2^2} + \frac{\Vert a_j \Vert_2^2}{\Vert A \Vert_F^2-\Vert a_j \Vert_2^2} \right)  \right)x^2_i \\
			&-
			\sum_{i=1}^{m}\sum\limits_{\substack{ j=1\\j > i}}^m
			\frac{2\left\langle a_i,a_j\right\rangle}{\Vert A \Vert_F^2}
			\left(\frac{\Vert a_i \Vert_2^2}{\Vert A \Vert_F^2-\Vert a_i \Vert_2^2} + \frac{\Vert a_j \Vert_2^2}{\Vert A \Vert_F^2-\Vert a_j \Vert_2^2} \right)
			x_ix_j \\
			=&
			\sum_{1\leq i<j\leq m}
			\frac{1}{\Vert A \Vert_F^2}
			\left(\frac{\Vert a_i \Vert_2^2}{\Vert A \Vert_F^2-\Vert a_i \Vert_2^2} + \frac{\Vert a_j \Vert_2^2}{\Vert A \Vert_F^2-\Vert a_j \Vert_2^2} \right)\left(\|a_j\|^2_2x_i^2+\|a_i\|^2_2x_j^2-2\langle a_i,a_j\rangle x_ix_j\right).
		\end{aligned}
		$$
		Similar to the discussion in the proof of Proposition \ref{propMN}, we can conclude that for any nonzero vector \( x = (x_1, x_2, \dots, x_m)^\top \in \mathbb{R}^m \), we have
		$$x^\top \left(M-2\left(I- \frac{AA^\top}{\|A\|_F^2}\right)\right) x > 0$$
		as desired.
	\end{proof}
	
}

\subsection{Omitted proofs in Section \ref{section4}}

\begin{proof}[Proof of Proposition \ref{prob-full}]
	We prove the result by contradiction. Suppose \( \dim(\Pi_k) < 2 \). Then there exists a scalar \( \lambda \in \mathbb{R} \) such that
	$
	x^k - z^k = \lambda (x^k - x^{k-1}).
	$
	If \( \lambda = 0 \), then \( x^k = z^k \), which contradicts the assumption \( z^k \neq x^k \). Hence, we must have \( \lambda \neq 0 \).
	Since the algorithm ensures that
	$
	\langle x^{k}-x^{k-1}, x^k-x_*^0\rangle=0,
	$
	it follows that
	$$
	\langle x^{k}-z^{k}, x^k-x_*^0\rangle=0.
	$$
	Recall  that $z^k = x^k - 2 u_ka_{i_{k_1}} - 2 v_ka_{i_{k_2}}$, we get
	$$
	\begin{aligned}
		0=\langle  u_ka_{i_{k_1}} + v_ka_{i_{k_2}}, x^k-x_*^0\rangle=u_k \langle  a_{i_{k_1}}, x^k-x_*^0\rangle+v_k \langle  a_{i_{k_2}}, x^k-x_*^0\rangle.
	\end{aligned}
	$$
	Without loss of generality,   we assume that  
	$
	\frac{ \langle a_{i_{k_2}}, x^k \rangle - b_{i_{k_2}} }{ \|a_{i_{k_2}}\|_2 }$ and $\frac{ \langle a_{i_{k_1}}, x^k \rangle - b_{i_{k_1}} }{ \|a_{i_{k_1}}\|_2 }
	$ 
	have the same sign.
	Using the fact that \( \langle a_i, x_*^0 \rangle = b_i \) for any $i\in[m]$, and the definitions
	$
	u_k = \frac{\langle a_{i_{k_1}}, x^k \rangle - b_{i_{k_1}}}{\|a_{i_{k_1}}\|_2^2} $ and $
	v_k = \frac{\langle a_{i_{k_2}}, x^k \rangle - b_{i_{k_2}} - 2 \langle a_{i_{k_2}}, a_{i_{k_1}} \rangle u_k}{\|a_{i_{k_2}}\|_2^2},
	$ we can rewrite the above as
	$$
	\begin{aligned}
		0=&\frac{(\langle a_{i_{k_1}}, x^k \rangle - b_{i_{k_1}})^2}{\|a_{i_{k_1}}\|_2^2}+\frac{(\langle a_{i_{k_2}}, x^k \rangle - b_{i_{k_2}})^2}{\|a_{i_{k_2}}\|_2^2}\\
		&-2\frac{ \langle a_{i_{k_2}}, a_{i_{k_1}} \rangle  }{\|a_{i_{k_1}}\|_2\|a_{i_{k_2}}\|_2}\cdot\frac{ \langle a_{i_{k_2}}, x^k \rangle - b_{i_{k_2}}}{\|a_{i_{k_2}}\|_2}\cdot\frac{\langle a_{i_{k_1}}, x^k \rangle - b_{i_{k_1}}}{\|a_{i_{k_1}}\|_2}
		\\
		\geq &\frac{(\langle a_{i_{k_1}}, x^k \rangle - b_{i_{k_1}})^2}{\|a_{i_{k_1}}\|_2^2}+\frac{(\langle a_{i_{k_2}}, x^k \rangle - b_{i_{k_2}})^2}{\|a_{i_{k_2}}\|_2^2}-2 \frac{ \langle a_{i_{k_2}}, x^k \rangle - b_{i_{k_2}}}{\|a_{i_{k_2}}\|_2}\cdot\frac{\langle a_{i_{k_1}}, x^k \rangle - b_{i_{k_1}}}{\|a_{i_{k_1}}\|_2}\\
		=&\left(\frac{\langle a_{i_{k_1}}, x^k \rangle - b_{i_{k_1}}}{\|a_{i_{k_1}}\|_2}-\frac{\langle a_{i_{k_2}}, x^k \rangle - b_{i_{k_2}}}{\|a_{i_{k_2}}\|_2} \right)^2,
	\end{aligned}
	$$
	where the inequality follows from the Cauchy-Schwarz inequality. Hence, we have 
	\begin{equation}\label{xie-proof-0604-1}
		\frac{\langle a_{i_{k_1}}, x^k \rangle - b_{i_{k_1}}}{\|a_{i_{k_1}}\|_2}=\frac{\langle a_{i_{k_2}}, x^k \rangle - b_{i_{k_2}}}{\|a_{i_{k_2}}\|_2}.
	\end{equation}
	Substituting this back, we find
	$$	0=2\frac{(\langle a_{i_{k_1}}, x^k \rangle - b_{i_{k_1}})^2}{\|a_{i_{k_1}}\|_2^2}-2\frac{ \langle a_{i_{k_2}}, a_{i_{k_1}} \rangle  }{\|a_{i_{k_1}}\|_2\|a_{i_{k_2}}\|_2}\cdot\frac{(\langle a_{i_{k_1}}, x^k \rangle - b_{i_{k_1}})^2}{\|a_{i_{k_1}}\|_2^2}.$$
	This leads to $\frac{(\langle a_{i_{k_1}}, x^k \rangle - b_{i_{k_1}})^2}{\|a_{i_{k_1}}\|_2^2}=0$ or $\frac{ \langle a_{i_{k_2}}, a_{i_{k_1}} \rangle  }{\|a_{i_{k_1}}\|_2\|a_{i_{k_2}}\|_2}=1$. 
	If 
	$
	\frac{(\langle a_{i_{k_1}}, x^k \rangle - b_{i_{k_1}})^2}{\|a_{i_{k_1}}\|_2^2} = 0,
	$
	then it follows from \eqref{xie-proof-0604-1} that
	$
	\frac{(\langle a_{i_{k_2}}, x^k \rangle - b_{i_{k_2}})^2}{\|a_{i_{k_2}}\|_2^2} = 0
	$. Consequently, we have \( u_k = v_k = 0 \), which implies \( z^k = x^k \), contradicting the assumption \( z^k \neq x^k \).
	On the other hand, if
	$
	\frac{\langle a_{i_{k_2}}, a_{i_{k_1}} \rangle}{\|a_{i_{k_1}}\|_2 \|a_{i_{k_2}}\|_2} = 1,
	$
	then the vectors \( a_{i_{k_1}} \) and \( a_{i_{k_2}} \) are colinear, meaning the corresponding hyperplanes \( \mathcal{H}_{i_{k_1}} \) and \( \mathcal{H}_{i_{k_2}} \) coincide. In this case, the reflection step yields \( z^k = x^k \), again contradicting the assumption \( z^k \neq x^k \).
	Therefore, our initial assumption that \( \dim(\Pi_k) < 2 \) must be false. Hence, \( \dim(\Pi_k) = 2 \).
\end{proof}

\begin{proof}[Proof of Proposition \ref{xie-0604-p}]
	Without loss of generality, assume that \(a_1 \neq 0\), and define the index set
	\[
	\mathcal{I} := \{ i \in [m] \mid a_i \text{ is linearly independent of } a_1 \}.
	\]
	Since \(\operatorname{rank}(A) \geq 2\), the set \(\mathcal{I}\) is nonempty.
	For any \(i \in \mathcal{I}\), using the same argument as in the proof of Proposition~\ref{prob-full}, we obtain the following two identities:
	$$
	\frac{\langle a_{1}, \tilde{x} \rangle - b_{1}}{\|a_{1}\|_2}=\frac{\langle a_{i}, \tilde{x} \rangle - b_{i}}{\|a_{i}\|_2}
	$$
	and
	$$
	0=2\frac{(\langle a_{1}, \tilde{x} \rangle - b_1)^2}{\|a_1\|_2^2}-2\frac{ \langle a_{1}, a_{i} \rangle  }{\|a_{1}\|_2\|a_{i}\|_2}\cdot\frac{(\langle a_{1}, \tilde{x} \rangle - b_{1})^2}{\|a_{1}\|_2^2}.
	$$
	Since \(a_1\) and \(a_i\) are linearly independent, it follows that \(\frac{\langle a_1, a_i \rangle}{\|a_1\|_2 \|a_i\|_2} < 1\).  Hence we know that $\frac{\langle a_{1}, \tilde{x} \rangle - b_{1}}{\|a_{1}\|_2}=\frac{\langle a_{i}, \tilde{x} \rangle - b_{i}}{\|a_{i}\|_2}=0$, i.e., $\langle a_{1}, \tilde{x} \rangle = b_{1}$ and  $\langle a_{i}, \tilde{x} \rangle = b_{i}$ for all $i\in\mathcal{I}$. 
	Now consider any \(j \in [m] \setminus \mathcal{I}\). Since \(a_j\) is linearly dependent on \(a_1\), there exists a scalar \(\lambda_j \in \mathbb{R}\) such that \(a_j = \lambda_j a_1\). Then
	$
	\langle a_j, \tilde{x} \rangle = \lambda_j \langle a_1, \tilde{x} \rangle = \lambda_j b_1 = b_j,
	$
	where the last equality follows from the consistency of the system \(Ax = b\). Hence, \(\langle a_j, \tilde{x} \rangle = b_j\) for all \(j \in [m] \setminus \mathcal{I}\).
	Combining the above, we conclude that \(\langle a_i, \tilde{x} \rangle = b_i\) for all \(i \in [m]\), i.e., \(\tilde{x}\) is a solution to the linear system \(Ax = b\).
\end{proof}

Next, we will prove Theorem \ref{main-AmRDR}. First, we introduce a useful lemma.
\begin{lemma} \label{tilde}
	Let $\{x^k\}_{k\geq0}$ and $\{z^k\}_{k\geq0}$ be the sequence generated by Algorithm \ref{alg:AmRDR} and $x_*^0=A^{\dagger}b+(I-A^\dagger A)x^0$. Then 
	$$
	\|x^{k+1} - \tilde{x}^{k+1}\|_2^2 =\cos^2 \theta_k \|\tilde{x}^{k+1} - x_*^0\|_2^2,
	$$ 
	where $\tilde{x}^{k+1}=\frac{1}{2}x^k+\frac{1}{2}z^k$, and \(\theta_k\) is the angle between the vectors \(\tilde{x}^{k+1} - x_*^0\) and \(\zeta_k= \langle z^k - x^k,\, x^k - x^{k-1} \rangle (z^k - x^k) - \|z^k - x^k\|_2^2 (x^k - x^{k-1})\).
\end{lemma}
\begin{proof}
	Since $\Vert \zeta_k\Vert^2_2 =\Vert z^k-x^k\Vert_2^2(\Vert z^k-x^k\Vert_2^2\Vert x^k-x^{k-1}\Vert_2^2-\langle z^k-x^k,x^k-x^{k-1}\rangle^2) $ and $\text{dim}(\Pi_k)=2 (k\geq 1)$, we know that $\zeta_k\neq 0$. Let $l_k\coloneqq \frac{\langle \tilde{x}^{k+1} - x_*^0,\zeta_k\rangle}{\Vert \zeta_k\Vert^2_2}$ and note that $\tilde{x}^{k+1}=\frac{1}{2}x^k+\frac{1}{2}z^k$, then we have 
	\begin{equation*}
		\begin{aligned}
			l_k\zeta_k
			&=\frac{\langle x^k - x_*^0,\zeta_k\rangle+\frac{1}{2}\langle z^k - x^k,\zeta_k\rangle}{\Vert \zeta_k\Vert^2_2}\zeta_k
			=\frac{\langle x^k - x_*^0,\zeta_k\rangle}
			{\Vert \zeta_k\Vert^2_2}\zeta_k\\
			&=\frac{\langle x^k - x_*^0,\langle z^k-x^k,x^k-x^{k-1}\rangle (z^k-x^k)-\Vert z^k-x^k\Vert_2^2 (x^k-x^{k-1})\rangle}
			{\Vert z^k-x^k\Vert_2^2(\Vert z^k-x^k\Vert_2^2\Vert x^k-x^{k-1}\Vert_2^2-\langle z^k-x^k,x^k-x^{k-1}\rangle^2)}\zeta_k\\
			&\overset{(a)}{=}\frac{\langle z^k-x^k,x^k-x^{k-1}\rangle\langle x^k - x_*^0, z^k-x^k\rangle}
			{\Vert z^k-x^k\Vert_2^2(\Vert z^k-x^k\Vert_2^2\Vert x^k-x^{k-1}\Vert_2^2-\langle z^k-x^k,x^k-x^{k-1}\rangle^2)}\zeta_k\\
			&=\frac{\langle z^k-x^k,x^k-x^{k-1}\rangle^2\langle x^k - x_*^0, z^k-x^k\rangle}
			{\Vert z^k-x^k\Vert_2^2(\Vert z^k-x^k\Vert_2^2\Vert x^k-x^{k-1}\Vert_2^2-\langle z^k-x^k,x^k-x^{k-1}\rangle^2)} (z^k-x^k)\\
			&\quad\ -\underbrace{\frac{\langle z^k-x^k,x^k-x^{k-1}\rangle\langle x^k - x_*^0, z^k-x^k\rangle}
				{\Vert z^k-x^k\Vert_2^2\Vert x^k-x^{k-1}\Vert_2^2-\langle z^k-x^k,x^k-x^{k-1}\rangle^2}}_{\beta_k}(x^k-x^{k-1}),
		\end{aligned}
	\end{equation*}
	where $(a)$ follows from $\langle x^k-x_*^0,x^k-x^{k-1}\rangle=0$ and the last equality follows from the definition of $\zeta_k$.	
	Since 
	\begin{equation*}
		\frac{\langle x^k-z^k,x^k-x_*^0\rangle}{\Vert z^k-x^k\Vert_2^2} 
		=\frac{\frac{1}{2}(\Vert z^k-x^k\Vert_2^2+\Vert x^k-x_*^0\Vert_2^2-\Vert z^k-x_*^0\Vert_2^2)}{\Vert z^k-x^k\Vert_2^2}
		= \frac{1}{2},
	\end{equation*}
	we can obtain that 
	
	\begin{equation*}
		\begin{aligned}
			\tilde{x}^{k+1} & -l_k\zeta_k \\
			=& x^k+
			\left(\frac{1}{2} -\frac{\langle z^k-x^k,x^k-x^{k-1}\rangle^2\langle x^k - x_*^0, z^k-x^k\rangle}
			{\Vert z^k-x^k\Vert_2^2(\Vert z^k-x^k\Vert_2^2\Vert x^k-x^{k-1}\Vert_2^2-\langle z^k-x^k,x^k-x^{k-1}\rangle^2)}\right) (z^k-x^k)\\
			& + \beta_k(x^k-x^{k-1}) \\
			=&x^k+\left(\frac{1}{2} -\frac{\langle z^k-x^k,x^k-x^{k-1}\rangle^2(-\frac{1}{2}\Vert z^k-x^k\Vert_2^2)}
			{\Vert z^k-x^k\Vert_2^2(\Vert z^k-x^k\Vert_2^2\Vert x^k-x^{k-1}\Vert_2^2-\langle z^k-x^k,x^k-x^{k-1}\rangle^2)}\right) (z^k-x^k)
			\\
			& +\beta_k(x^k-x^{k-1})\\
			=&x^k+\frac{1}{2} \frac{\Vert z^k-x^k\Vert_2^2\Vert x^k-x^{k-1}\Vert_2^2}
			{\Vert z^k-x^k\Vert_2^2\Vert x^k-x^{k-1}\Vert_2^2-\langle z^k-x^k,x^k-x^{k-1}\rangle^2} (z^k-x^k)
			+\beta_k(x^k-x^{k-1})\\
			=&x^k+\frac{\Vert x^k-x^{k-1}\Vert_2^2 \langle x^k-z^k,x^k-x_*^0\rangle}
			{\Vert z^k-x^k\Vert_2^2\Vert x^k-x^{k-1}\Vert_2^2-\langle z^k-x^k,x^k-x^{k-1}\rangle^2} (z^k-x^k)
			+\beta_k(x^k-x^{k-1})\\
			=&x^k+\alpha_k(z^k-x^k)
			+\beta_k(x^k-x^{k-1}) 
			= x^{k+1}.
		\end{aligned}
	\end{equation*}
	Thus we have 
	$$
	\|x^{k+1} - \tilde{x}^{k+1}\|_2^2 =l_k^2\|\zeta_k\|_2^2 = \frac{\langle \tilde{x}^{k+1} - x_*^0,\zeta_k\rangle^2}{\Vert \zeta_k\Vert^2_2}=\cos^2 \theta_k \|\tilde{x}^{k+1} - x_*^0\|_2^2
	$$
	as desired.
\end{proof}

Now, we are ready to prove Theorem \ref{main-AmRDR}.
\begin{proof}[Proof of Theorem \ref{main-AmRDR}]
	Since $x^{k+1} - x_*^0$ is orthogonal to the direction space of $\Pi_k$ and $x^{k+1}, \tilde{x}^{k+1} \in \Pi_k$, the Pythagorean Theorem implies that
	\begin{equation*}
		\|x^{k+1} - x_*^0\|_2^2 = \|\tilde{x}^{k+1} - x_*^0\|_2^2 - \|x^{k+1} - \tilde{x}^{k+1}\|_2^2.
	\end{equation*}
	From Lemma \ref{tilde}, we know that 
	$$
	\|x^{k+1} - \tilde{x}^{k+1}\|_2^2  = \cos^2 \theta_k \|\tilde{x}^{k+1} - x_*^0\|_2^2,
	$$
	where $\theta_k$ denotes the angle between $\tilde{x}^{k+1} - x_*^0$ and $\zeta_k$.
	Hence we have 
	$$
	\begin{aligned}
		\mathbb{E}\left[\|x^{k+1}-x_*^0\|_2^2 \mid  x^k\right]&=\mathbb{E}_{(i,j)\in \mathcal{Q}_k}\left[\|x^{k+1}-x_*^0\|_2^2 \mid  x^k\right]
		\\&=\mathbb{E}_{(i,j)\in \mathcal{Q}_k}\left[\|\tilde{x}^{k+1}-x_*^0\|_2^2  \mid x^k\right]-\mathbb{E}_{(i,j)\in \mathcal{Q}_k}\left[\|x^{k+1}-\tilde{x}^{k+1}\|_2^2 \mid x^k\right]
		\\
		&=\mathbb{E}_{(i,j)\in \mathcal{Q}_k}\left[(1-\cos^2\theta_k)\|\tilde{x}^{k+1}-x_*^0\|_2^2 \mid  x^k\right]
		\\
		&\leq
		\left( 1-\gamma_k \right) \mathbb{E}_{(i,j)\in \mathcal{Q}_k}\left[\|\tilde{x}^{k+1}-x_*^0\|_2^2 \mid  x^k\right]\\
		&\leq \left( 1-\gamma_k \right)\mathbb{E}\left[\|\tilde{x}^{k+1}-x_*^0\|_2^2 \mid  x^k\right],
	\end{aligned}
	$$
	where the first inequality follows from the definition of $\gamma_k$.
Combining the preceding inequality with the one-step estimate in Theorem~\ref{main-IRDR} for $\alpha=\frac{1}{2}$, we obtain, for Strategy I,
	$$
	\mathbb{E}\left[\|x^{k+1}-x_*^0\|_2^2 \mid  x^k\right]
	\leq
	\left( 1-\gamma_k \right)
	\left(1-\frac{\sigma_{\min}^2({ S_M}^{\frac{1}{2}}A)}{\|A\|^2_F}
	\right)
	\|x^k - x_*^0\|_2^2 ,
	$$
	and, for Strategy II,
	$$
	\mathbb{E}\left[\|x^{k+1}-x_*^0\|_2^2 \mid  x^k\right]
	\leq
	\left( 1-\gamma_k \right)
	\left(
	1-2
	\frac{\sigma_{\min}^2(N^{\frac{1}{2}}A)}
	{ \Vert A \Vert_F^4 -\Vert AA^{\top} \Vert_F^2}
	\right)
	\|x^k - x_*^0\|_2^2 ,
	$$
	where ${ S_M}$ and $N$ are given by \eqref{sym-skew-decomp} and \eqref{definition-N}, respectively.
\end{proof}

\end{document}